\documentclass[a4paper,12pt]{article}

\usepackage[left=2cm,right=2cm, top=2cm,bottom=3cm,bindingoffset=0cm]{geometry}

\usepackage{verbatim}
\usepackage{amsmath}
\usepackage{amsthm}
\usepackage{amssymb}
\usepackage{delarray}
\usepackage{cite}
\usepackage{hyperref}
\usepackage{mathrsfs}
\usepackage{tikz}
\usetikzlibrary{patterns}
\usepackage{caption}
\DeclareCaptionLabelSeparator{dot}{. }
\newcommand{\be}{\beta}
\newcommand{\ga}{\gamma}
\newcommand{\de}{\delta}
\newcommand{\la}{\lambda}
\newcommand{\om}{\omega}

\newcommand{\eps}{\varepsilon}
\newcommand{\vv}{\varphi}

\theoremstyle{plain}

\numberwithin{equation}{section}

\newtheorem{thm}{Theorem}[section]
\newtheorem{lem}[thm]{Lemma}
\newtheorem{prop}[thm]{Proposition}

\theoremstyle{definition}

\newtheorem{ip}[thm]{Inverse Problem}

\newtheorem{defin}[thm]{Definition}

\theoremstyle{remark}

\newtheorem{remark}[thm]{Remark}

\DeclareMathOperator*{\Res}{Res}

 \allowdisplaybreaks

\begin{document}

\begin{center}
{\Large\bf Uniform stability of higher-order \\[0.2cm]
inverse spectral problems}
\\[0.5cm]
{\bf Natalia P. Bondarenko}
\end{center}

\vspace{0.5cm}

{\bf Abstract.} In this paper, the reconstruction of a linear differential operator of arbitrary order $n \ge 2$ is studied
by using two types of spectral characteristics: (i) eigenvalues and weight numbers, (ii) $(2n-2)$ spectra.
We prove the unconditional uniform stability of these inverse problems, generalizing the results of Savchuk and Shkalikov [Funct. Anal. Appl. 44 (2010), no. 4, 270--285] to $n > 2$. 
Furthermore, we for the first time obtain sufficient conditions of solvability for the higher-order inverse problem by $(2n-2)$ spectra. By applying our main results, we get new theorems on the necessary and sufficient conditions of solvability and on the uniform stability of the inverse problems for $n = 3$ and $n = 4$. Our approach is based on the method of spectral mappings, which provides a constructive solution of the inverse problems.

\medskip


\medskip


\vspace{1cm}

\section{Introduction} \label{sec:intr}

Consider the differential equation
\begin{align} \nonumber
\ell_n(y) := & y^{(n)} + \sum_{k = 0}^{\lfloor n/2\rfloor - 1} (\tau_{2k}(x) y^{(k)})^{(k)} \\ \label{eqv} + & \sum_{k = 0}^{\lfloor (n-1)/2\rfloor - 1}  \bigl((\tau_{2k+1}(x) y^{(k)})^{(k+1)} + (\tau_{2k+1}(x) y^{(k+1)})^{(k)}\bigr) = \la y, \: x \in (0,1),
\end{align}
where $n \ge 2$, the notation $\lfloor a \rfloor$ means rounding a real number $a$ down, $\tau_{\nu}$ belongs to the Sobolev spaces $W_2^{\nu}[0,1]$ for $\nu = \overline{0,n-2}$, and $\la$ is the spectral parameter.

Equation \eqref{eqv} can be equivalently represented in the form
\begin{equation} \label{eqp}
\ell_n(y) = y^{(n)} + \sum_{s = 0}^{n-2} p_s(x) y^{(s)} = \lambda y, \quad x \in (0,1),
\end{equation}
where $p_s \in W_2^s[0,1]$, $s = \overline{0,n-2}$. However, we use the divergent form \eqref{eqv}, since we focus on the ``self-adjoint'' case  $i^{n + \nu} \tau_{\nu}(x) \in \mathbb R$.

For $k = \overline{1,n-1}$, denote by $\{ \la_{l,k} \}_{l \ge 1}$ and $\{ \mu_{l,k} \}_{l \ge 1}$ the eigenvalues (counting with multiplicities) of the boundary value problems $\mathcal L_k$ and $\mathcal M_k$, respectively, for equation \eqref{eqv} with the corresponding boundary conditions
\begin{align} \label{bc}
\mathcal L_k \colon & \quad y^{(j-1)}(0) = 0, \quad j = \overline{1,k}, \qquad \qquad \quad \:\:\: y^{(s-1)}(1) = 0, \quad s = \overline{1,n-k}, \\ \label{bcmu}
\mathcal M_k \colon & \quad y^{(j-1)}(0) = 0, \quad j = \overline{1,k-1},\,k+1, \quad y^{(s-1)}(1) = 0, \quad s = \overline{1,n-k}.
\end{align}

We study the inverse spectral problems that consist in the reconstruction of the functions $\{ \tau_{\nu} \}_{\nu = 0}^{n-2}$ from the two types of spectral data:

\smallskip

(i) the eigenvalues $\{ \la_{l,k} \}_{l \ge 1, k = \overline{1,n-1}}$ and the related weight numbers $\{ \be_{l,k} \}_{l \ge 1, \, k = \overline{1,n-1}}$, which are rigorously defined in Section~\ref{sec:main}; 

\smallskip

(ii) the system of $(2n-2)$ spectra $\{ \la_{l,k}, \mu_{l,k} \}_{l \ge 1, \, k = \overline{1,n-1}}$.

\smallskip

The first problem statement goes back to the seminal paper \cite{GL51} by Gelfand and Levitan for $n = 2$ and to the studies \cite{Leib62, Leib66, Leib71, Leib72} by Leibenzon for $n > 2$. The solution of the second problem, which generalizes the famous Borg inverse problem by two spectra \cite{Borg46}, was considered by Yurko \cite[Section~2.5.3]{Yur02} for $n > 2$. In this paper, we prove theorems on the uniform stability of the inverse problems, which have no analogs for higher orders $n > 2$. Moreover, we for the first time find sufficient conditions for the existence of solution for the inverse problem by the system of spectra.

The most complete results in the inverse spectral theory were obtained for the second-order Sturm-Liouville operators (see monographs \cite{Mar77, Lev84, PT87, FY01, Krav20} and references therein). Inverse problems for the third-order differential operators have applications to integration of the nonlinear Boussinesq equation (see \cite{McK81, Jur91, BK24}). Inverse problems for $n = 4$ arise in geophysics \cite{Bar74} and vibration theory \cite{Glad05}. However, transformation operators, which played an important role in the development of the inverse spectral theory for $n = 2$, appeared to be ineffective for higher orders. Only some results were achieved under the analyticity requirements for the coefficients of differential equations (see \cite{Sakh61, Khach83, Mal94}). 

Leibenzon has developed another approach that allowed him to prove the uniqueness \cite{Leib66} and to create a constructive procedure for the recovery of the coefficients $\{ p_s \}_{s = 0}^{n-2}$ of equation \eqref{eqp} from the spectral data $\{ \la_{l,k}, \be_{l,k} \}_{l \ge 1, \, k =\overline{1,n-1}}$. Furthermore, in \cite{Leib72}, the necessary and sufficient conditions for solvability of this inverse problem were obtained. However, as was pointed out by Yurko \cite[Section~2.3.7]{Yur02}, Leibenzon's spectral data uniquely determine the coefficients only if the separation condition $\{ \la_{l,k} \}_{l \ge 1} \cap \{ \la_{l,k+1} \}_{l \ge 1} = \varnothing$ is satisfied. As an alternative, Yurko introduced the Weyl-Yurko matrix, a spectral characteristic that ensures the uniqueness of the solution to the inverse problem for higher-order operators without any additional restrictions on their spectra. This facilitated the development of a general inverse problem theory for equation \eqref{eqp} with $p_s \in W_2^{s + \nu}$, $s = \overline{0,n-2}$, $\nu \ge 0$, on both a finite interval and the half-line (see \cite{Yur92, Yur00, Yur02}). Meanwhile, the case of the whole line requires a different approach (see \cite{Beals85, Beals88}). In recent years, the method of \cite{Yur02} was transferred to arbitrary order differential operators with distribution coefficients (see \cite{Bond21, Bond22, Bond23-loc, Bond23-res, Bond24, GBY25}). In particular, the most challenging issue regarding the necessary and sufficient conditions for the solvability of inverse spectral problems was resolved. For third- and fourth-order operators, those conditions take their simplest form, involving only the asymptotics of the spectral data and their basic structural properties (see \cite{Bond23-res, Bond24}). We also mention recent studies \cite{ATU25, Zol25, BK26} on inverse scattering and spectral theory for the third-order differential operators. Some specific kinds of inverse problems for $n = 4$ were considered in \cite{McL86, PK97, CPS98}. 
 
In recent years, significant progress has been achieved in the investigation of the uniform stability of inverse spectral problems for second-order differential, integro-differential, and functional-differential operators \cite{SS10, SS13, Hryn11, Bond25, But21, BD22, Kuz23}. Savchuk and Shkalikov \cite{SS10} were the first to describe the set of spectral data of the Sturm-Liouville equation for which the inverse spectral mapping is uniformly bounded and, as a consequence, the unconditional uniform stability of the inverse problem is satisfied. However, to the best of the author's knowledge, there were no such kind of results for higher orders. This paper aims to fill this gap.

To reconstruct the parameters $\{ \tau_{\nu} \}_{\nu = 0}^{n-2}$ of equation \eqref{eqv} from the spectral data $\{ \la_{l,k}, \be_{l,k} \}$,
we develop the approach of the previous studies \cite{Yur02, Yur00, Bond22, Bond23-loc, Bond24}, based on the method of spectral mappings. This method allows us to reduce the nonlinear inverse problem to a linear equation $(I - \tilde R(x)) \psi(x) = \tilde \psi(x)$ in the Banach space of bounded infinite sequences. In this paper, we apply a new modification of the main equation which makes the operator $\tilde R(x)$ and the right-hand side $\tilde \psi(x)$ continuous with respect to the spectral data $\la_{l,k}$ and $\be_{l,k}$. This modification was introduced in \cite{Bond25} and was crucial for proving the uniform stability of the inverse problem in the case $n = 2$. The subsequent analysis consists of three steps. First, we consider the general non-self-adjoint case and prove that the solution of the inverse problem is uniformly bounded under the constraint $\| (I - \tilde R(x))^{-1} \| \le K$, where $K$ is a positive constant. Second, basing on the uniform boundedness, we prove the uniform stability under the same assumptions. Third, we consider the self-adjoint case and find the restrictions on the spectral data that guaranty $\| (I - \tilde R(x))^{-1} \| \le K$. As a result, we get the unconditional uniform stability of the inverse spectral problem for equation \eqref{eqv} of arbitrary order $n$. Furthermore, we reduce the inverse problem by the eigenvalue set $\{ \la_{l,k}, \mu_{l,k} \}_{l \ge 1, \, k = \overline{1,n-1}}$ to the inverse problem by the spectral data $\{ \la_{l,k}, \be_{l,k} \}_{l \ge 1, \,k = \overline{1,n-1}}$ and transfer our main results to the first problem. Namely, we obtain sufficient conditions for the existence of solution and prove the unconditional uniform stability of the inverse spectral problem by the $(2n-2)$ spectra $\{ \la_{l,k}, \mu_{l,k} \}_{l \ge 1, \, k = \overline{1,n-1}}$.

One of the conditions of our main theorems  (Theorems~\ref{thm:uni}, \ref{thm:scsp}, and~\ref{thm:unisp}) is the asymptotics of the eigenvalues $\la_{l,k}$, $\mu_{l,k}$ and of the weight numbers $\be_{l,k}$. Derivation of explicit formulas for the coefficients in those asymptotics for higher orders $n$ is technically complicated. Therefore, we focus on the cases $n = 3, 4$ and find the corresponding coefficients by using the standard method from \cite{Nai68, Yur22} and symbolic computations in Python \cite{py}. Finally, we deduce new theorems on the necessary and sufficient conditions of solvability and on the uniform stability of inverse problems for $n = 3, 4$ from the main results for arbitrary $n$.

\medskip

The paper is organized as follows. In Section~\ref{sec:main}, we formulate the inverse spectral problems and the main theorems. 
Sections~\ref{sec:maineq}--\ref{sec:sa} are concerned with the inverse problem by eigenvalues and weight numbers.
In Section~\ref{sec:maineq}, the inverse problem is reduced to a linear equation in the Banach space of bounded infinite sequences. In Section~\ref{sec:bound}, we
consider the problem in the general non-self-adjoint case and prove the uniform boundedness of the inverse problem under the bound $\| (I - \tilde R(x))^{-1} \| \le K$ together with other suitable conditions. In Section~\ref{sec:stab}, we prove the uniform stability of the inverse problem under the assumptions of the previous section. In Section~\ref{sec:sa}, the unconditional uniform stability of the inverse problem is obtained in the self-adjoint case. In Section~\ref{sec:sp}, solvability and stability results are transferred to the inverse problem by the $(2n-2)$ spectra. 
In Sections~\ref{sec:3} and~\ref{sec:4}, we apply our main theorems to the special cases $n = 3$ and $n = 4$, respectively. In Appendix~\ref{app:Weyl}, asymptotic properties of solutions of equation \eqref{eqv} are described basing on previous studies \cite{Yur02, Bond21, Bond22, Bond23-loc}. In Appendix~\ref{app:calc}, we present the derivation of the spectral data asymptotics for $n = 3, 4$. In Appendix~\ref{app:schur}, we provide a discrete weighted version of Schur's test, which is used for proving an auxiliary estimate in Section~\ref{sec:sp}.

\smallskip

Throughout the paper, we use the following notations:
\begin{itemize}
\item $W_2^s[0,1]$ is the Sobolev space with the norm
$$
    \| y \|_{W_2^s[0,1]} = \left( \sum_{j = 0}^s \| y^{(j)} \|^2_{L_2[0,1]} \right)^{1/2}.
$$
\item $\de_{j,k}$ denotes the Kronecker delta.
\item Along with $\tau = \{ \tau_{\nu} \}_{\nu = 0}^{n-2}$, we consider other vector functions $\tilde \tau = \{ \tilde \tau_{\nu} \}_{\nu = 0}^{n-2}$, $\tilde{\tilde \tau} = \{ \tilde{\tilde \tau}_{\nu} \}_{\nu = 0}^{n-2}$, and $\tau_{(u)} = \{ \tau_{\nu,(u)} \}_{\nu = 0}^{n-2}$ ($u \ge 1$) of the same class.
We agree that, if a symbol $\ga$ denotes an object related to $\tau$, then the symbols $\tilde \ga$, $\tilde{\tilde \ga}$, and $\ga_{(u)}$ will denote the similar objects related to $\tilde \tau$, $\tilde{\tilde \tau}$, and $\tau_{(u)}$, respectively.
\item In estimates, the same symbol $C$ denotes various positive constants independent of $\la$, $x$, $n$, etc. The notation $C(A_1, A_2, \dots)$ means that the constant $C$ depends on the parameters $A_1$, $A_2$, \dots, e.g., $C(\Omega,\de)$.
\item $C_n^k = \dfrac{n!}{k!(n-k)!}$ are binomial coefficients.
\item The notations $\lfloor a \rfloor$ and $\lceil a \rceil$ are used for rounding a real number $a$ down and up, respectively.
\item $J = \bigl\{ (l,k) \colon l \ge 1, \, k = \overline{1,n-1} \bigr\}$.
\item $V = \bigl\{ (l,k,\eps) \colon (l,k) \in J, \eps = 0, 1 \bigr\}$.
\item $m$ is the Banach space of bounded infinite sequences $a = [a_v]_{v \in V}$ with the norm $\| a \|_m = \sup_{v \in V} |a_v|$.
\item $I$ is the identity operator in $m$.
\item For $Q > 0$, denote by $B_Q$ the set of vectors $\tau = \{ \tau_{\nu} \}_{\nu = 0}^{n-2}$ satisfying $\| \tau_{\nu} \|_{W_2^{\nu}} \le Q$, $\nu = \overline{0,n-2}$.
\end{itemize}

\section{Main results} \label{sec:main}

Let us begin with some preliminaries.

\begin{defin} \label{def:W}
Introduce the following classes of vectors $\tau = \{ \tau_{\nu} \}_{\nu = 0}^{n-2}$:
\begin{itemize}
\item $\mathbf{W}$ the class of $\tau$ such that $\tau_{\nu} \in W_2^{\nu}[0,1]$.
\item $\mathbf{W}^+$ is the class of $\tau \in \mathbf{W}$ such that $i^{n + \nu} \tau_{\nu}$ are real-valued for $\nu = \overline{0,n-2}$.
\item $\mathbf{W}_{simp}$ is the class of $\tau \in \mathbf{W}$ such that the eigenvalues $\{ \la_{l,k} \}_{l \ge 1}$ of the corresponding problems $\mathcal L_k$ satisfy the following simplifying assumptions:

\smallskip

\textbf{(A-1)} For each fixed $k \in \{ 1, \dots, n-1 \}$, the eigenvalues $\{ \la_{l,k} \}_{l \ge 1}$ are simple.

\smallskip

\textbf{(A-2)} $\{ \la_{l,k} \}_{l \ge 1} \cap \{ \la_{l,k+1} \}_{l \ge 1} = \varnothing$ for $k = \overline{1,n-2}$.
\item $\mathbf{W}_{simp}^+ := \mathbf{W}^+ \cap \mathbf{W}_{simp}$.
\end{itemize}
\end{defin}

For $k = \overline{1,n}$, denote by $\Phi_k(x,\la)$ the so-called \textit{Weyl solution} of equation \eqref{eqv} satisfying the boundary conditions
\begin{equation} \label{bcPhi}
\Phi_k^{(j-1)}(0,\la) = \de_{k,j}, \quad j = \overline{1,k}, \qquad \Phi_k^{(s-1)}(1,\la) = 0, \quad s = \overline{1,n-k}.
\end{equation}

Define the \textit{Weyl functions} as $M_k(\la) := \Phi_k^{(k)}(0,\la)$, $k = \overline{1,n-1}$. The functions $\Phi_k^{(j-1)}(x,\la)$ for each fixed $x \in [0,1]$, $j = \overline{1,n}$ and $M_k(\la)$ are meromorphic in $\la$, and their poles coincide with the eigenvalues $\{ \la_{l,k} \}_{l \ge 1}$ (see \cite[Theorem~2.1.1]{Yur02}). Under the assumptions (A-1) and (A-2), these poles are simple and
\begin{equation} \label{defbe}
\be_{l,k} := \Res_{\la = \la_{l,k}} M_k(\la) \ne 0, \quad (l,k) \in J,
\end{equation}
where $J := \{ (l,k) \colon l \ge 1, \, k = \overline{1,n-1} \}$.

We call $\{ \be_{l,k} \}_J$ \textit{the weight numbers} and the sequence $\{ \la_{l,k}, \be_{l,k} \}_J$ \textit{the spectral data} of $\tau \in \mathbf{W}$. Consider the following inverse spectral problem.

\begin{ip} \label{ip:main}
Given the spectral data $\{ \la_{l,k}, \be_{l,k} \}_J$, find $\tau = \{ \tau_{\nu} \}_{\nu = 0}^{n-2} \in \mathbf{W}_{simp}$.
\end{ip}

The solution of Inverse Problem~\ref{ip:main} in the class $\mathbf{W}_{simp}$ is unique (see \cite{Leib62}, \cite[Theorem~2.5.1]{Yur02} and \cite[Corollary~1]{Bond23-loc}). However, the assumptions (A-1) and (A-2) are crucial for uniqueness. In order to determine $\tau$ in the general case, one has to specify the Weyl-Yurko matrix (see \cite[Theorem~2.1.2]{Yur02}). Necessary and sufficient conditions for solvability of Inverse Problem~\ref{ip:main} in the class $\tau \in \mathbf{W}_{simp}^+$ have been obtained in \cite{Bond24}. Definitions~\ref{def:S} and~\ref{def:S+} below describe the properties that are sufficient for a sequence $\{ \la_{l,k}, \be_{l,k} \}_J$ to be the spectral data of some vector $\tau \in \mathbf{W}_{simp}^+$.

\begin{defin} \label{def:S}
For $\tilde \tau \in \mathbf{W}$, denote by $\mathcal S = \mathcal S(\tilde \tau)$ the set of sequences $\{ \la_{l,k}, \be_{l,k} \}_J$ of complex numbers satisfying the assumptions (A-1), (A-2), $\be_{l,k} \ne 0$ for all $(l,k) \in J$, and
\begin{equation} \label{defXi}
\Xi := \left( \sum_{l = 1}^{\infty} \biggl( \sum_{k = 1}^{n-1} \bigl( |\la_{l,k} - \tilde \la_{l,k}| + l^{-1} |\be_{l,k} - \tilde \be_{l,k}| \bigr) \biggr)^2 \right)^{1/2} < \infty.
\end{equation}
\end{defin}

Note that the spectral data possess the following asymptotics (see \cite[formulas (2.3.1) and (2.3.2)]{Yur02} and \cite[formulas (10) and (11)]{Bond22}):
\begin{align} \label{asymptla}
\la_{l,k} & = (-1)^{n-k} \bigl(c_{0,k} l^n + c_{1,k} l^{n-1} + c_{2,k} l^{n-2} + \dots + c_{l,k}  + \varkappa_{l,k} \bigr), \quad \{ \varkappa_{l,k} \} \in l_2, \\ \label{asymptbe}
\be_{l,k} & = -n \la_{l,k} \bigl( 1 + d_{1,k} l^{-1} + \dots + d_{n-1,k} l^{-(n-1)} + l^{-(n-1)} \eta_{l,k} \bigr), \quad \{ \eta_{l,k} \} \in l_2,
\end{align}
where $c_{j,k}$ and $d_{j,k}$ are complex constants that depend on $\tau$, excluding $c_{1,k}$ and
\begin{equation} \label{c0k}
c_{0,k} = \biggl( \frac{\pi}{\sin \frac{\pi k}{n}} \biggr)^n.
\end{equation}
Therefore, the condition \eqref{defXi} in Definition~\ref{def:S} means that $\{ \la_{l,k}, \be_{l,k} \}_J$ satisfy the relations \eqref{asymptla} and \eqref{asymptbe} with the coefficients $c_{j,k}$ and $d_{j,k}$ equal to the respective coefficients $\tilde c_{j,k}$ and $\tilde d_{j,k}$ from the spectral data asymptotics of $\tilde \tau$.
One can obtain explicit formulas for the coefficients $c_{j,k}$ and $d_{j,k}$ by using $\{ \tau_{\nu} \}_{\nu = 0}^{n-2}$. Examples for $n = 3$ and $n = 4$ are provided in Sections~\ref{sec:3} and~\ref{sec:4}, respectively. However, for higher orders $n$, derivation of explicit formulas is technically complicated. Therefore, we will formulate our main result in terms of the asymptotical proximity of values $\{ \la_{l,k}, \be_{l,k} \}_J$ to the spectral data $\{ \tilde \la_{l,k}, \tilde \be_{l,k} \}_J$ of a so-called model vector $\tilde \tau$.

\begin{defin} \label{def:S+}
For $\tilde \tau \in \mathbf{W}^+$, denote by $\mathcal S^+ = \mathcal S^+(\tilde\tau)$ the set of the sequences $\{ \la_{l,k}, \be_{l,k} \}_J \in \mathcal S$ that additionally satisfy the conditions
\begin{align} \label{sa}
& \la_{l,k} = (-1)^n \overline{\la_{l,n-k}}, \quad \be_{l,k} = (-1)^n \overline{\be_{l,n-k}}, \quad (l,k) \in J, \\ \label{addhyp1}
& \text{if} \:\: n = 2p \colon  \quad (-1)^{p+1}\be_{l,p} > 0, \quad l \ge 1, \\ \label{addhyp2}
& \text{if} \:\: n = 2p+1 \colon  \quad (-1)^{p+1} \mbox{Re}\, \la_{l,p} > 0,  \quad l \ge 1.
\end{align}
\end{defin}

Note that the relations \eqref{sa} and \eqref{addhyp1} hold for the spectral data of any $\tau \in \mathbf{W}_{simp}^+$ (see \cite{Bond24}).

The following proposition provides sufficient conditions for the solvability of Inverse Problem~\ref{ip:main}.

\begin{prop}[\hspace*{-5pt}\cite{Bond24}, Theorem 1] \label{prop:sc}
Let $\tilde \tau = \{ \tilde \tau_{\nu} \}_{\nu = 0}^{n-2} \in \mathbf{W}_{simp}^+$. Then every sequence $\{ \la_{l,k}, \be_{l,k} \}_J$ of $\mathcal S^+(\tilde \tau)$ is the spectral data of a unique vector $\tau = \{ \tau_{\nu} \}_{\nu = 0}^{n-2} \in \mathbf{W}_{simp}^+$.
\end{prop}

Strictly speaking, in \cite{Bond24}, the case $\tau_{\nu} \in W_2^{\nu-1}[0,1]$ ($\nu = \overline{0,n-2}$) was considered. Higher degree of smoothness $\tau_{\nu} \in W_2^{\nu}[0,1]$ corresponds to one more term in the asymptotic relations \eqref{asymptla} and \eqref{asymptbe} comparing with \cite{Bond24}. 

\begin{defin} \label{def:SOe+}
For $\Omega, \de > 0$ and $\tilde \tau \in \mathbf{W}_{simp}^+$, denote by $\mathcal S^+_{\Omega, \de} = \mathcal S^+_{\Omega,\de}(\tilde \tau)$ the set of sequences $\{ \la_{l,k}, \be_{l,k} \}_J$ in $\mathcal S^+(\tilde \tau)$ satisfying the additional conditions: 
\begin{gather}
\label{bbound1}
\begin{array}{c}
|\la_{l,k} - \la_{s,k}| \ge \de, \quad l \ne s, \, k = \overline{1,n-1}, \qquad
|\la_{l,k} - \la_{s,k+1}| \ge \de, \quad k = \overline{1,n-2}, \\ |\be_{l,k}| \ge \de, \quad k = \overline{1,n-1}, \qquad |\mbox{Re} \, \la_{l,p}| \ge \de \:\: \text{if $n = 2p+1$},
\end{array}
\end{gather}
and $\Xi \le \Omega$, where $\Xi$ is defined in \eqref{defXi}.
\end{defin}

The main result of this paper is the following theorem on the uniform stability of Inverse Problem~\ref{ip:main}. 

\begin{thm} \label{thm:uni}
Suppose that $\Omega > 0$, $\de > 0$, and $\tilde \tau \in \mathbf{W}_{simp}^+$. Then, for any sequences $\{ \la_{l,k,(1)}, \be_{l,k,(1)} \}_J$ and $\{ \la_{l,k,(2)}, \be_{l,k,(2)} \}_J$ in $\mathcal S^+_{\Omega,\de}(\tilde \tau)$, the corresponding solutions $\tau_{(1)}$ and $\tau_{(2)}$ of Inverse Problem~\ref{ip:main} satisfy the estimates
\begin{equation} \label{esttau}
\| \tau_{\nu,(1)} - \tau_{\nu,(2)} \|_{W_2^{\nu}[0,1]} \le C(\Omega,\de) Z, \quad \nu = \overline{0,n-2}, 
\end{equation}
where
\begin{equation} \label{defZ}
Z := \left( \sum_{l = 1}^{\infty} \biggl( \sum_{k = 1}^{n-1} \bigl( |\la_{l,k,(1)} - \la_{l,k,(2)}| + l^{-1} |\be_{l,k,(1)} - \be_{l,k,(2)}| \bigr)\biggr)^2\right)^{1/2}.
\end{equation}
\end{thm}

Note that the value $Z$ in \eqref{esttau} is finite, because $Z^2 \le \Xi_{(1)}^2 + \Xi_{(2)}^2$ due to \eqref{defXi} and \eqref{defZ}.

Proceed to the inverse problem by the system of $(2n-2)$ spectra of the problems $\mathcal L_k$ \eqref{eqv}, \eqref{bc} and $\mathcal M_k$ \eqref{eqv}, \eqref{bcmu} for $k = \overline{1,n-1}$. We call the sequence of the corresponding eigenvalues $\{ \la_{l,k}, \mu_{l,k} \}_J$ \textit{the eigenvalue set} of $\tau = \{ \tau_{\nu} \}_{\nu = 0}^{n-2}$.

\begin{ip} \label{ip:sp}
Given the eigenvalue set $\{ \la_{l,k}, \mu_{l,k} \}_J$, find $\tau = \{ \tau_{\nu} \}_{\nu = 0}^{n-2} \in \mathbf{W}_{simp}$.
\end{ip}

As shown in Section~\ref{sec:sp}, Inverse Problem~\ref{ip:sp} is easily reduced to Inverse Problem~\ref{ip:main}, which implies the uniqueness of recovering $\tau \in \mathbf{W}_{simp}$ from the eigenvalue set (see \cite[Theorem~2.5.5]{Yur02}). We show that this reduction is stable and obtain for Inverse Problem~\ref{ip:sp} results analogous to Proposition~\ref{prop:sc} and Theorem~\ref{thm:uni}. We begin with the following lemma, which describes some structural properties of the eigenvalue set.

\begin{lem} \label{lem:nscsp}
The eigenvalue set $\{ \la_{l,k}, \mu_{l,k} \}_J$ of $\tau \in \mathbf{W}_{simp}^+$ satisfy the following conditions:
\begin{gather} \label{sepsp}
\{ \la_{l,k} \}_{l \ge 1} \cap \{ \mu_{l,k} \}_{l \ge 1} = \varnothing, \quad k = \overline{1,n-1}, \\ \label{sasp}
\la_{l,k} = (-1)^n \overline{\la_{l,n-k}}, \quad \mu_{l,k} = (-1)^n \overline{\mu_{l,n-k}}, \quad (l,k) \in J, \\ \label{inter}
\text{if} \:\: n = 2p \colon \quad (-1)^p \mu_{l,p} < (-1)^p \la_{l,p} < (-1)^p \mu_{l+1,p}, \quad l \ge 1. 
\end{gather}
\end{lem}

The following definition and theorem show that the conditions of Lemma~\ref{lem:nscsp} together with suitable asymptotics, the separation conditions, and \eqref{addhyp2} in the odd-order case are sufficient for a sequence $\{  \la_{l,k}, \mu_{l,k} \}_J$ to be the eigenvalue set of some $\tau \in \mathbf{W}_{simp}^+$.

\begin{defin} \label{def:E+}
For $\tilde \tau \in \mathbf{W}^+$, denote by $\mathcal E^+ = \mathcal E^+(\tilde \tau)$ the set of sequences $\{ \la_{l,k}, \mu_{l,k} \}_J$ of complex numbers satisfying the assumptions (A-1), (A-2), \eqref{addhyp2}, \eqref{sepsp}, \eqref{sasp}, \eqref{inter}, and
\begin{equation}
\label{defTheta}
\Theta := \left( \sum_{l = 1}^{\infty} \biggl( \sum_{k = 1}^{n-1} \bigl( |\la_{l,k} - \tilde \la_{l,k}| + \ln(l + 1) |\mu_{l,k} - \tilde \mu_{l,k}| \bigr) \biggr)^2 \right)^{1/2} < \infty.
\end{equation}
\end{defin}

Note that the eigenvalues $\{ \mu_{l,k} \}_{l \ge 1}$ are not required to be simple.

\begin{thm} \label{thm:scsp}
Let $\tilde \tau = \{ \tilde \tau_{\nu} \}_{\nu = 0}^{n-2} \in \mathbf{W}_{simp}^+$. Then every sequence $\{ \la_{l,k}, \mu_{l,k} \}_J \in \mathcal E^+(\tilde \tau)$ is the eigenvalue set of a unique vector $\tau = \{ \tau_{\nu} \}_{\nu = 0}^{n-2} \in \mathbf{W}_{simp}^+$.
\end{thm}

Next, let us formulate the analog of Definition~\ref{def:SOe+} for eigenvalue sets.

\begin{defin} \label{def:EOde+}
For $\Omega, \de > 0$ and $\tilde \tau \in \mathbf{W}_{simp}^+$, denote by $\mathcal E_{\Omega,\de}^+ = \mathcal E_{\Omega,\de}^+(\tilde \tau)$ the set of sequences $\{ \la_{l,k}, \mu_{l,k} \}_J$ in $\mathcal E^+(\tilde \tau)$ satisfying the additional conditions
\begin{gather*}
|\la_{l,k} - \la_{s,k}| \ge \de, \quad l \ne s, \, k = \overline{1,n-1}, \qquad  |\la_{l,k} - \la_{s,k+1}| \ge \de, \quad k = \overline{1,n-2}, \\
|\la_{l,k}| \ge \de, \quad |\la_{l,k} - \mu_{s,k}| \ge \de, \quad k = \overline{1,n-1}, \qquad
|\mbox{Re} \, \la_{l,p}| \ge \de \:\: \text{if $n = 2p+1$},
\end{gather*}
and $\Theta \le \Omega$, where $\Theta$ is defined in \eqref{defTheta}.
\end{defin}

Relying on Theorem~\ref{thm:uni}, we obtain the unconditional uniform stability of Inverse Problem~\ref{ip:sp}:

\begin{thm} \label{thm:unisp}
Suppose that $\Omega > 0$, $\de > 0$, and $\tilde \tau \in \mathbf{W}_{simp}^+$. Then, for any sequences $\{ \la_{l,k,(1)}, \mu_{l,k,(1)} \}_J$ and $\{ \la_{l,k,(2)}, \mu_{l,k,(2)} \}_J$ in $\mathcal E_{\Omega,\de}^+(\tilde \tau)$, the corresponding solutions $\tau_{(1)}$ and $\tau_{(2)}$ of Inverse Problem~\ref{ip:sp} satisfy the estimates
\begin{equation} \label{unisp}
\| \tau_{\nu,(1)} - \tau_{\nu,(2)} \|_{W_2^{\nu}[0,1]} \le C(\Omega,\de) X, \quad \nu = \overline{0, n-2},
\end{equation}
where
\begin{equation*} 
X := \left( \sum_{l = 1}^{\infty} \biggl( \sum_{k = 1}^{n-1} \bigl( |\la_{l,k,(1)} - \la_{l,k,(2)}| + \ln(l + 1) |\mu_{l,k,(1)} - \mu_{l,k,(2)}| \bigr)\biggr)^2\right)^{1/2}.
\end{equation*}
\end{thm}

Theorems~\ref{thm:uni} and~\ref{thm:unisp} generalize the results of Savchuk and Shkalikov \cite{SS10, SS13} on the uniform stability of the inverse Sturm-Liouville problems to the case of arbitrary order differential operators. 

\begin{remark}
Observe that Proposition~\ref{prop:sc}, Theorems~\ref{thm:uni}, \ref{thm:scsp}, and~\ref{thm:unisp} have nonlocal nature. Indeed, the proximity of the sequences $\{ \la_{l,k}, \be_{l,k} \}_J$ and $\{ \la_{l,k}, \mu_{l,k} \}_J$ to the corresponding spectral data of the model vector $\tilde \tau$ is characterized by the values $\Xi$ \eqref{defXi} and $\Theta$ \eqref{defTheta}, respectively, which are assumed to be bounded but not necessarily small. This is an essential difference of Theorems~\ref{thm:uni} and~\ref{thm:unisp} from the local stability results of previous studies on higher-order inverse problems (see \cite[Corollary~7.1]{Leib72}, \cite[Theorems~2.3.2, 2.3.4 and~2.5.3]{Yur02}, \cite[Theorem~1]{Bond23-loc}).  
\end{remark}

\section{Main equation} \label{sec:maineq}

In this section, Inverse Problem~\ref{ip:main} is reduced to a linear equation in the Banach space of bounded infinite sequences. Our construction of the main equation is analogous to previous studies \cite{Yur02, Bond22, Bond23-loc, Bond24}, so we outline it briefly. However, we apply a new modification that appears in \cite{Bond25} for the $n=2$. This modification makes the components of the operator and the free term in the main equation continuous with respect to the spectral data (see Remark~\ref{rem:modif} for details). This feature is crucial for investigating the stability of the inverse problem.

Consider two vectors of coefficients $\tau = \{ \tau_{\nu} \}_{\nu = 0}^{n-2}$ and $\tilde \tau = \{ \tilde \tau_{\nu} \}_{\nu = 0}^{n-2}$ of $\mathbf{W}_{simp}$. The eigenvalues are assumed to be reordered so that, if $\la_{l,k} = \tilde \la_{s,k}$, then $l = s$. In addition, we assume that
\begin{equation} \label{assump}
\la_{l,k} \ne \tilde \la_{s,k\pm 1} \quad \text{for all $l,s \ge 1$ and $k$}.
\end{equation}

Introduce the notation
\begin{gather} \nonumber 
    V := \{ (l,k,\eps) \colon l \in \mathbb N, \, k = \overline{1,n-1}, \, \eps = 0, 1 \},  \\ \nonumber
    \la_{l,k,0} := \la_{l,k}, \quad \la_{l,k,1} := \tilde \la_{l,k}, \quad \be_{l,k,0} := \be_{l,k}, \quad \be_{l,k,1} := \tilde \be_{l,k}, \\ \label{defvv}
    \vv_{l,k,\eps}(x) := \Phi_{k+1}(x, \la_{l,k,\eps}), \quad \tilde \vv_{l,k,\eps}(x) := \tilde \Phi_{k+1}(x, \la_{l,k,\eps}), \quad (l,k,\eps) \in V.
\end{gather}

Recall that the function $\Phi_{k+1}(x,\la)$ is meromorphic in the $\la$-plane with the poles $\{ \la_{l,k+1} \}_{l \ge 1}$. In view of the assumptions (A-2) and \eqref{assump}, the points $\la_{l,k}$ and $\tilde \la_{l,k}$ are regular for $\Phi_{k+1}(x,\la)$, so the functions $\vv_{l,k,\eps}(x)$ are correctly defined. Similar arguments are valid for $\tilde \vv_{l,k,\eps}(x)$.

For $k = \overline{1,n}$, denote by $\Phi_k^{\star}(x, \la)$ the solution of the differential equation
\begin{align} \nonumber
\ell_n^{\star}(y) := & (-1)^n y^{(n)} + \sum_{k = 0}^{\lfloor n/2\rfloor - 1} (\tau_{2k}(x) y^{(k)})^{(k)} \\ \label{eqstar} - & \sum_{k = 0}^{\lfloor (n-1)/2\rfloor - 1}  \bigl((\tau_{2k+1}(x) y^{(k)})^{(k+1)} + (\tau_{2k+1}(x) y^{(k+1)})^{(k)}\bigr) = \la y, \quad x \in (0,1),
\end{align}
satisfying the boundary conditions \eqref{bcPhi}. Obviously, in the case $\tau \in \mathbf{W}^+$, we have $\Phi_k^{\star}(x,\la) = \overline{\Phi_k(x, (-1)^n \overline{\la})}$.

Define the functions
\begin{equation} \label{defD}
\tilde D_{k,k_0}(x, \mu, \la) = \begin{cases}
\int_0^x \tilde \Phi_k^{\star}(t, \mu) \tilde \Phi_{k_0}(t, \la) \, dt, & k + k_0 > n + 1, \\
\frac{(-1)^{k+1}}{\la - \mu} + \int_0^x \tilde \Phi_k^{\star}(t, \mu) \tilde \Phi_{k_0}(t, \la) \, dt, & k + k_0 = n + 1, \\
-\int_x^1 \tilde \Phi_k^{\star}(t, \mu) \tilde \Phi_{k_0}(t, \la) \, dt, & k + k_0 < n + 1.
\end{cases}
\end{equation}

In \cite[Corollary~2]{Bond22} and \cite[Proposition~3]{Bond24}, the following relation has been obtained:
\begin{equation} \label{relPhik}
\Phi_{k_0}(x, \la) = \tilde \Phi_{k_0}(x, \la) + \sum_{(l,k,\eps) \in V} (-1)^{\eps + n - k} \be_{l,k,\eps} \vv_{l,k,\eps}(x) \tilde D_{n-k+1,k_0}(x, \la_{l,k,\eps}, \la), \quad k_0 = \overline{1,n}.
\end{equation}

Let us discuss the correctness of the terms $\tilde D_{n-k+1,k_0}(x, \la_{l,k,\eps}, \la)$, which contain the functions $\tilde \Phi^{\star}_{n-k+1}(x,\la_{l,k,\eps})$. The eigenvalues of the problems $\mathcal L_k^{\star}$ for equation \eqref{eqstar} with the boundary conditions \eqref{bc} satisfy the relation $\la_{l,k}^{\star} = \la_{l,n-k}$, $l \ge 1$, $k = \overline{1,n-1}$ (see \cite[Lemma~1]{Bond24}). Therefore, the Weyl solution $\tilde \Phi^{\star}_{n-k+1}(x,\la)$ has the poles $\{ \tilde \la_{l,k-1} \}_{l \ge 1}$ if $k > 1$. The assumptions (A-2) and \eqref{assump} imply that $\la_{l,k}$ and $\tilde \la_{l,k}$ are regular points of $\tilde \Phi^{\star}_{n-k+1}(x,\la)$, so the functions $\tilde D_{n-k+1,k_0}(x, \la_{l,k,\eps}, \la)$ are correctly defined.

Basing on the relation~\eqref{relPhik}, we provide a new construction of the main equation for Inverse Problem~\ref{ip:main}, employing the approach of \cite{Bond25}. Put $\chi_{l,k} := l^{1-n} (\la_{l,k} - \tilde \la_{l,k})$ and introduce the matrices
$$
P_{l,k} := \begin{bmatrix}
\chi_{l,k} & 1 \\
0 & 1
\end{bmatrix}, \quad 
P_{l,k}^{-1} = \begin{bmatrix}
\chi_{l,k}^{-1} & -\chi_{l,k}^{-1} \\
0 & 1
\end{bmatrix} \:\: \text{if $\la_{l,k} \ne \tilde \la_{l,k}$}.
$$
Define the functions
\begin{equation} \label{defw}
w_{l,k}(x) := l^{-k} \exp(-xl \cot(k\pi/n)),
\end{equation}
which are related to the growth of $\vv_{l,k,\eps}(x)$: $|\vv_{l,k,\eps}(x)| \le C w_{l,k}(x)$ (see Lemma~\ref{lem:estvv}).

For $(l,k,\eps), (l_0,k_0,\eps_0) \in V$, denote
\begin{align} \nonumber
    \tilde G_{(l,k,\eps), (l_0, k_0, \eps_0)}(x) & :=  
    (-1)^{n-k} \be_{l,k,\eps} \tilde D_{n-k+1, k_0 + 1}(x, \la_{l,k,\eps}, \la_{l_0,k_0,\eps_0}), \\ \label{defdvv}
\dot \vv_{l,k,\eps}(x) & := l^{n-1}\frac{d}{d\la} \Phi_{k+1}(x,\la)\Big|_{\la = \la_{l,k,\eps}}, \\ \nonumber \dot{\tilde G}_{(l,k,\eps), (l_0,k_0,\eps_0)}(x) & := (-1)^{n-k} l^{l-1} \be_{l,k,\eps} \frac{d}{d \la} \tilde D_{n-k+1, k_0 + 1}(x, \la_{l,k,\eps}, \la)\Big|_{\la = \la_{l_0,k_0,\eps_0}}.
\end{align}

Introduce the functions
\begin{align} 
\begin{bmatrix} \label{defpsi1}
\psi_{l,k,0}(x) \\ \psi_{l,k,1}(x)
\end{bmatrix} & := 
w_{l,k}^{-1}(x) P_{l,k}^{-1}
\begin{bmatrix}
\vv_{l,k,0}(x) \\ \vv_{l,k,1}(x)
\end{bmatrix}, \: \la_{l,k} \ne \tilde \la_{l,k}, \\ \label{defpsi2}
\begin{bmatrix}
\psi_{l,k,0}(x) \\ \psi_{l,k,1}(x)
\end{bmatrix} & := 
w_{l,k}^{-1}(x)
\begin{bmatrix}
\dot{\vv}_{l,k,1}(x) \\ \vv_{l,k,1}(x)
\end{bmatrix}, \: \la_{l,k} = \tilde \la_{l,k}, 
\end{align}
\vspace*{-5pt}
\begin{multline} \label{defR1}
\begin{bmatrix}
\tilde R_{(l_0,k_0,0),(l,k,0)}(x) & \tilde R_{(l_0,k_0,0),(l,k,1)}(x) \\
\tilde R_{(l_0,k_0,1),(l,k,0)}(x) & \tilde R_{(l_0,k_0,1),(l,k,1)}(x)
\end{bmatrix} \\ := 
\frac{w_{l,k}(x)}{w_{l_0,k_0}(x)}
P_{l_0,k_0}^{-1}
\begin{bmatrix}
\tilde G_{(l,k,0),(l_0,k_0,0)}(x) & -\tilde G_{(l,k,1),(l_0,k_0,0)}(x) \\
\tilde G_{(l,k,0),(l_0,k_0,1)}(x) & -\tilde G_{(l,k,1),(l_0,k_0,1)}(x)
\end{bmatrix} P_{l,k}, \quad \la_{l_0,k_0} \ne \tilde \la_{l_0,k_0}.
\end{multline}
\vspace*{-5pt}
\begin{multline} \label{defR2}
\begin{bmatrix}
\tilde R_{(l_0,k_0,0),(l,k,0)}(x) & \tilde R_{(l_0,k_0,0),(l,k,1)}(x) \\
\tilde R_{(l_0,k_0,1),(l,k,0)}(x) & \tilde R_{(l_0,k_0,1),(l,k,1)}(x)
\end{bmatrix} \\ := 
\frac{w_{l,k}(x)}{w_{l_0,k_0}(x)}
\begin{bmatrix}
\dot{\tilde G}_{(l,k,0),(l_0,k_0,1)}(x) & -\dot{\tilde G}_{(l,k,1),(l_0,k_0,1)}(x) \\
\tilde G_{(l,k,0),(l_0,k_0,1)}(x) & -\tilde G_{(l,k,1),(l_0,k_0,1)}(x)
\end{bmatrix} P_{l,k}, \quad \la_{l_0,k_0} = \tilde \la_{l_0,k_0}.
\end{multline}
Analogously to $\psi_{l,k,\eps}(x)$, define $\tilde \psi_{l,k,\eps}(x)$ by using $\tilde \vv_{l,k,\eps}(x)$ instead of $\vv_{l,k,\eps}(x)$.

It follows from the relation \eqref{relPhik} that
\begin{equation} \label{relpsi}
\psi_{l_0,k_0,\eps_0}(x) = \tilde \psi_{l_0,k_0,\eps_0}(x) + \sum_{(l,k,\eps) \in V} \tilde R_{(l_0,k_0,\eps_0), (l,k,\eps)}(x) \psi_{l,k,\eps}(x), \quad (l_0,k_0,\eps_0) \in V, \: x \in [0,1].
\end{equation}

The relations \eqref{relpsi} can be treated as a system of main equations for solving Inverse Problem~\ref{ip:main}. Indeed, suppose that the spectral data $\{ \la_{l,k}, \be_{l,k} \}_J$ and the model vector $\tilde \tau$ are given. Then, one can construct the functions $\tilde \psi_{l,k,\eps}(x)$ and $\tilde R_{(l_0,k_0,\eps_0), (l,k,\eps)}(x)$ by using \eqref{defpsi1}, \eqref{defpsi2} and \eqref{defR1}, \eqref{defR2}, respectively. By solving the system \eqref{relpsi}, the solution $\{ \psi_{l,k,\eps}(x) \}_V$ is obtained, and then one can find
\begin{equation} \label{findvv}
\begin{bmatrix}
\vv_{l,k,0}(x) \\ \vv_{l,k,1}(x)
\end{bmatrix} = w_{l,k}(x) P_{l,k} 
\begin{bmatrix}
\psi_{l,k,0}(x) \\ \psi_{l,k,1}(x)
\end{bmatrix}
\end{equation}
by inverting \eqref{defpsi1} and \eqref{defpsi2}.
Using $\vv_{l,k,\eps}(x)$, one obtains $\Phi_{k_0}(x,\la)$ ($k_0 = \overline{1,n}$) by the formula \eqref{relPhik} and recovers $\{ \tau_{\nu} \}_{\nu = 0}^{n-2}$.

\begin{remark} \label{rem:modif}
The difference of the system \eqref{relpsi} from its analogs in previous studies is contained in the terms $\psi_{l,k,0}(x)$ and $\tilde R_{(l_0,k_0,0),(l,k,\eps)}$ for $\la_{l,k} = \tilde \la_{l,k}$ and $\la_{l_0,k_0} = \tilde \la_{l_0,k_0}$, respectively. This case was treated in another way in \cite{Bond22} and excluded from consideration in \cite{Yur02, Bond24}. We use the derivatives by $\la$ in \eqref{defpsi2} and \eqref{defR2}, which makes the corresponding components $\psi_{l,k,0}(x)$ and $\tilde R_{(l_0,k_0,0),(l,k,\eps)}(x)$ continuous with respect to the eigenvalues. This property is important for studying the uniform stability for the inverse problem. 
In view of \eqref{findvv} and \eqref{relPhik}, the components $\psi_{l,k,0}(x)$ for $\la_{l,k} = \tilde \la_{l,k}$ do not influence $\Phi_{k_0}(x,\la)$ and therefore are uniquely determined by the other components $\psi_{l,k,\eps}(x)$. Consequently, the solvability of the system \eqref{relpsi} is equivalent to the solvability of the main equation (6) in \cite{Bond24}. 
\end{remark}

Let us represent the system \eqref{relpsi} in the operator form.
Put $v = (l,k,\eps)$ and $v_0 = (l_0,k_0,\eps_0)$. 
Denote
\begin{equation} \label{defxi}
\xi_l  := \sum_{k = 1}^{n-1} \left( l^{-(n-1)} |\la_{l,k} - \tilde \la_{l,k}| + l^{-n} |\be_{l,k} - \tilde \be_{l,k}| \right), \quad l \ge 1.
\end{equation}

There hold the estimates
\begin{equation} \label{estpsiR}
|\psi_v(x)|, \, |\tilde \psi_v(x)| \le C, \quad |\tilde R_{v_0,v}(x)| \le \frac{C \xi_l}{|l - l_0| + 1}, \quad v, v_0 \in V, \quad x \in [0,1].
\end{equation}

Recall that $m$ is the Banach space of bounded infinite sequences $a = [a_v]_{v \in V}$ with the norm $\| a \|_m = \sup_{v \in V} |a_v|$.
In view of \eqref{estpsiR}, the vectors $\psi(x) := [\psi_v(x)]_V$ and $\tilde \psi(x) := [\tilde \psi_v(x)]_V$ belong to $m$ for each fixed $x \in [0,1]$.
Moreover, the linear operator $\tilde R(x)$ given by the rule 
\begin{equation} \label{opR}
(\tilde R(x) a)_{v_0} = \sum_{v \in V} \tilde R_{v_0, v}(x) a_v, \quad a \in m,
\end{equation}
is bounded from $m$ to $m$:
\begin{equation} \label{estop}
\| \tilde R(x) \|_{m \to m} = \sup_{v_0 \in V} \sum_{v \in V} |\tilde R_{v_0,v}(x)| \le C \left( \sum_{l= 1}^{\infty} \xi_l^2 \right)^{1/2} < \infty, \quad x \in [0,1].
\end{equation}

Thus, the system \eqref{relpsi} can be represented in the operator form:
\begin{equation} \label{maineq}
(I - \tilde R(x)) \psi(x) = \tilde \psi(x), \quad x \in [0,1],
\end{equation}
where $I$ is the identity operator.

Due to \cite[Theorem~1]{Bond22}, the operator $(I - \tilde R(x))$ has a bounded inverse on $m$ for each fixed $x \in [0,1]$, so equation \eqref{maineq} has the unique solution $\psi(x)$. The operator $(I - \tilde R(x))$ plays an important role in the proof of uniform stability theorems in the next sections.

\section{Uniform bounds} \label{sec:bound}

In this section, we consider Inverse Problem~\ref{ip:main} in the non-self-adjoint case and prove that its solution is uniformly bounded under some a priori conditions on the spectral data.

Assume that $\tilde \tau \in \mathbf{W}_{simp}$ is fixed. Then, for any sequence $\{ \la_{l,k}, \be_{l,k} \}_J \in \mathcal S(\tilde \tau)$ satisfying the conditions \eqref{assump}, one can construct the linear bounded operator $\tilde R(x)$ in the Banach space $m$ following the arguments of Section~\ref{sec:maineq}.

\begin{defin} \label{def:SOe}
For $\Omega, K, \de > 0$ and $\tilde \tau \in \mathbf{W}_{simp}$, denote by $\mathcal S_{\Omega,K,\de} = \mathcal S_{\Omega, K, \de}(\tilde \tau)$ the set of sequences $\{ \la_{l,k}, \be_{l,k} \}_J$ in $\mathcal S(\tilde \tau)$ such that 
\begin{equation} \label{Xixi}
\Xi = \left(\sum\limits_{l = 1}^{\infty} \bigl(l^{n-1}\xi_l\bigr)^2\right)^{1/2} \le \Omega
\end{equation}
and
\begin{equation} \label{bbound}
|\la_{l,k} - \tilde \la_{s,k\pm1}| \ge \de \quad \text{for all $l,s\ge1$ and $k$},
\end{equation}
moreover, the corresponding operator $(I - \tilde R(x))$ has a bounded inverse on $m$ for each fixed $x \in [0,1]$, and
\begin{equation} \label{boundK}
\bigl\| (I - \tilde R(x))^{-1}\bigr\|_{m \to m} \le K, \quad x \in [0,1].
\end{equation}
\end{defin}

Note that the conditions \eqref{bbound} imply \eqref{assump}, which is essential for the operator $\tilde R(x)$ to be correctly defined. Furthermore, the bounds \eqref{estop}, \eqref{Xixi}, and \eqref{bbound}
imply the uniform bound for the operator:
\begin{equation} \label{uniR}
\| \tilde R(x) \|_{m \to m} \le C(\Omega,\de), \quad x \in [0,1].
\end{equation}

The goal of this section is to prove the following theorem.

\begin{thm} \label{thm:unibound}
Every sequence $\{ \la_{l,k}, \be_{l,k} \}_J \in \mathcal S_{\Omega,K,\de}$ is the spectral data of a unique vector $\tau \in \mathbf{W}_{simp}$ and
\begin{equation} \label{unibound}
\| \tau_{\nu} \|_{W_2^{\nu}[0,1]} \le C(\Omega,K,\de), \quad \nu = \overline{0,n-2}.
\end{equation}
\end{thm}

Theorem~\ref{thm:unibound} is proved analogously to Theorem~2 in \cite{Bond23-loc}. The most principal differences are as follows:
\begin{itemize}
\item The coefficients $\tau_{\nu}$ in this paper have one more derivative than in \cite{Bond23-loc}, and so do auxiliary functions such as $\psi_{l,k,\eps}(x)$, $\vv_{l,k,\eps}(x)$, etc.
\item We need estimates to be uniform on $\mathcal S_{\Omega,K,\de}$, which is achieved using the conditions \eqref{defXi}, \eqref{bbound}, and \eqref{boundK} in Definition~\ref{def:SOe}.
\end{itemize}

Thus, we outline the proof of Theorem~\ref{thm:unibound} briefly, not elaborating into technical details.

\begin{proof}[Proof of Theorem~\ref{thm:unibound}]
Consider the main equation \eqref{maineq} constructed by $\{ \la_{l,k}, \be_{l,k} \}_J$ in $\mathcal S_{\Omega,K,\de}$. Since the bounded inverse operator $(I - \tilde R(x))^{-1}$ exists for each $x \in [0,1]$, we can find the unique solution of \eqref{maineq}:
$$
\psi(x) = [\psi_{l,k,\eps}(x)]_V := (I - \tilde R(x))^{-1} \tilde \psi(x).
$$

Next, we construct the functions $[\vv_{l,k,\eps}(x)]_V$ by \eqref{findvv} and obtain the following estimates for them analogously to \cite[Lemma~2]{Bond23-loc}.

\begin{lem} \label{lem:vv}
For $(l,k,\eps) \in V$, we have $\vv_{l,k,\eps} \in C^{n-1}[0,1]$ and 
\begin{gather*}
|\vv_{l,k,\eps}^{(\nu)}(x)| \le C l^{\nu} w_{l,k}(x), \quad
|\vv_{l,k,0}^{(\nu)}(x) - \vv_{l,k,1}^{(\nu)}(x)| \le C l^{\nu} w_{l,k}(x) \xi_l, \quad \nu = \overline{0,n-1}, \\
|\vv_{l,k,\eps}(x) - \tilde \vv_{l,k,\eps}(x)| \le C w_{l,k}(x) \chi_l, \\
|\vv_{l,k,0}(x) - \vv_{l,k,1}(x) - \tilde \vv_{l,k,0}(x) + \tilde \vv_{l,k,1}(x)| \le C w_{l,k}(x) \chi_l \xi_l, \\
\left.\begin{array}{c}
|\vv_{l,k,\eps}^{(\nu)}(x) - \tilde \vv_{l,k,\eps}^{(\nu)}(x)| \le C l^{\nu - 1}w_{l,k}(x), \\
|\vv_{l,k,0}^{(\nu)}(x) - \vv_{l,k,1}^{(\nu)} - \tilde \vv_{l,k,0}^{(\nu)}(x) + \tilde \vv_{l,k,1}^{(\nu)}(x)| \le C  l^{\nu - 1} w_{l,k}(x) \xi_l, 
\end{array} \right\}
\quad \nu = \overline{1,n-1},
\end{gather*}
where $x \in [0,1]$, $C = C(\Omega,K,\de)$, and
$$
\chi_l := \left( \sum_{k = 1}^{\infty} \frac{1}{k^2 (|l-k|+1)^2} \right)^{1/2}, \quad \{ \chi_l \}_{l \ge 1} \in l_2.
$$
\end{lem}

In \cite[Section~2.3.3]{Yur02} and \cite[Section~5]{Bond23-loc}, the following reconstruction formulas have been derived for the coefficients in the representation \eqref{eqp} for $\ell_n(y)$:
\begin{align} \nonumber
p_s(x) = & \, \tilde p_s(x) - \left( t_{n,s}(x) + (-1)^{n-s} T_{0,n-s-1}(x) \right) \\ \label{recp} & + 
\sum_{j = 0}^{n-s-3} \sum_{r = j}^{n-s-3} (-1)^r C_r^j \tilde p_{r+s+1}^{(r-j)}(x) T_{0,j}(x) - \sum_{r = s+1}^{n-2} p_r(x) t_{r,s}(x), 
\end{align}
for $s = n-2, n-3, \dots, 1, 0$,
where $C_n^k = \frac{n!}{k!(n-k)!}$ are the binomial coefficients and
\begin{align} 
\label{defT}
T_{j_1, j_2}(x) & := \sum_{(l,k,\eps) \in V} (-1)^{\eps} \vv_{l,k,\eps}^{(j_1)}(x) \tilde \eta_{l,k,\eps}^{(j_2)}(x), \\ \label{deft}
t_{r,s}(x) & := \sum_{u = s}^{r-1} C_r^{u+1} C_u^s T_{r-u-1,u-s}(x), \\ \label{defeta}
\tilde \eta_{l,k,\eps}(x) & := (-1)^{n-k} \be_{l,k,\eps} \tilde \Phi^{\star}_{n-k+1}(x, \la_{l,k,\eps}).
\end{align}

The summation of several series $T_{j_1,j_2}(x)$ in \eqref{recp}, \eqref{deft} and below is understood in the following sense:
\begin{equation} \label{sum}
\sum_{(l,k,\eps) \in V} a_{l,k,\eps} + \sum_{(l,k,\eps) \in V} b_{l,k,\eps} = \sum_{l = 1}^{\infty} \sum_{k = 1}^{n-1} (a_{l,k,0} + a_{l,k,1} + b_{l,k,0} + b_{l,k,1}).
\end{equation}

Similarly to \cite[Lemma~3]{Bond23-loc} and Lemma~\ref{lem:estvv}, we get the following properties for $\eta_{l,k,\eps}(x)$.

\begin{lem} \label{lem:eta}
For $(l,k,\eps) \in V$, we have $\eta_{l,k,\eps} \in C^{n-1}[0,1]$ and
$$
|\tilde \eta_{l,k,\eps}^{(\nu)}(x)| \le C(\Omega,\de) l^{\nu} w_{l,k}^{-1}(x), \quad
|\tilde \eta_{l,k,0}^{(\nu)}(x) - \tilde \eta_{l,k,1}^{(\nu)}(x)| \le C(\Omega,\de) l^{\nu} w_{l,k}^{-1}(x) \xi_l, \quad \nu = \overline{0,n-1}.
$$
\end{lem}

Using Lemmas~\ref{lem:vv} and~\ref{lem:eta} together with \eqref{defT}, we conclude that, for $j_1 + j_2 = n - s - 1$, $s \in \{ 1, 2, \dots, n-1\}$, the series $T_{j_1,j_2}(x)$ converges in $W_2^s[0,1]$ and
$$
\| T_{j_1,j_2}(x) \|_{W_2^s[0,1]} \le C(\Omega,K,\de).
$$
This together with \eqref{recp} and \eqref{deft} imply that $p_s \in W_2^s[0,1]$ and
\begin{equation} \label{estp}
\| p_s \|_{W_2^s[0,1]} \le C(\Omega,K,\de).
\end{equation}

The coefficients $\{ \tau_{\nu} \}_{\nu = 0}^{n-2}$ of \eqref{eqv} can be found from $\{ p_s \}_{s = 0}^{n-2}$ as follows:
\begin{multline} \label{rectau}
2^j \tau_s = p_s - (1-j) \tau_{s+1}' - \sum_{k = \lceil (s+1)/2\rceil}^{\min \{ s, \lfloor n/2\rfloor-1 \}} C_k^{s-k} \bigl( \tau_{2k}^{(2k-s)} + \tau_{2k+1}^{(2k+1-s)}\bigr) \\ - \sum_{k = \lceil s/2\rceil}^{\min\{ s, \lfloor (n-1)/2\rfloor\}-1} 2 C_k^{s-k-1} \tau_{2k+1}^{(2k+1-s)}, \quad s = n-2, n-3, \dots, 1, 0,
\end{multline}
where $j = 0$ if $s$ is even and $j = 1$ is $s$ is odd, $\tau_{n-1} = 0$. 
This together with \eqref{estp} imply \eqref{esttau}.

Analogously to \cite[Lemma~6]{Bond23-loc}, we show that the spectral data of $\tau = \{ \tau_{\nu} \}_{\nu = 0}^{n-2}$ equal $\{ \la_{l,k}, \be_{l,k} \}_J$, which concludes the proof of Theorem~\ref{thm:unibound}.
\end{proof}

\section{Uniform stability} \label{sec:stab}

In this section, we continue the discussion of Inverse Problem~\ref{ip:main} in the general non-self-adjoint case, started in the previous section. We consider two sets of spectral data and estimate the differences of the corresponding differential expression parameters. Specifically, we prove the following theorem on the uniform stability of Inverse Problem~\ref{ip:main} in the general non-self-adjoint case.

\begin{thm} \label{thm:unistab2}
Let $\tilde \tau \in \mathbf{W}_{simp}$ and $\Omega, K, \de > 0$ be fixed. Then, for any two spectral data sets $\{ \la_{l,k,(1)}, \be_{l,k,(1)} \}_J$ and $\{ \la_{l,k,(2)}, \be_{l,k,(2)} \}_J$ in $\mathcal S_{\Omega,K,\de}(\tilde \tau)$ the corresponding solutions $\tau_{(1)} = \{ \tau_{\nu,(1)} \}_{\nu = 0}^{n-2}$ and $\tau_{(2)} = \{ \tau_{\nu,(2)} \}_{\nu = 0}^{n-2}$ of Inverse Problem~\ref{ip:main} satisfy the estimate
\begin{equation} \label{diftau}
\| \tau_{\nu,(1)} - \tau_{\nu,(2)} \|_{W_2^{\nu}[0,1]} \le C(\Omega,K,\de) Z, \quad \nu = \overline{0,n-2},
\end{equation}
where $Z$ was defined in \eqref{defZ}.
\end{thm}

\begin{proof}
Under the hypothesis of the theorem, the solutions $\tau_{(1)}$ and $\tau_{(2)}$ uniquely exist by virtue of Theorem~\ref{thm:unibound}. Furthermore, the vectors $\tau_{(1)}$ and $\tau_{(2)}$ belong to $\mathbf{W}_{simp} \cap B_Q$, where 
$$
B_Q = \bigl\{ \tau \in \mathbf{W} \colon \| \tau_{\nu} \|_{W_2^{\nu}} \le Q, \, \nu = \overline{0,n-2}\bigr\},
$$
and $Q > 0$ depends only on $\Omega$, $K$, and $\de$. 

Recall that the solutions $\tau_{\nu,(1)}$ and $\tau_{\nu,(2)}$ are constructed by using the linear relations \eqref{recp} and \eqref{rectau}. Therefore, in order to prove \eqref{diftau}, it is sufficient to obtain the estimate
\begin{gather} \label{difT}
\| T_{j_1,j_2,(1)} - T_{j_1,j_2,(2)} \|_{W_2^s[0,1]} \le C(\Omega, K, \de) Z, \\ \label{ind}
j_1,j_2 \ge 0, \quad j_1 + j_2 = n-s-1, \quad s \in \{ 1,2, \dots, n-1 \}.
\end{gather}

Represent $T_{j_1,j_2}$ as follows:
\begin{align} \label{sumTj}
T_{j_1,j_2}(x) & = \tilde T_{j_1,j_2}(x) + \hat T_{j_1,j_2}(x), \\ \nonumber
\tilde T_{j_1,j_2}(x) & := \sum_{(l,k) \in J} \bigl( \tilde \vv_{l,k,0}^{(j_1)}(x)\tilde \eta_{l,k,0}^{(j_2)}(x) - \tilde \vv_{l,k,1}^{(j_1)}(x) \tilde \eta_{l,k,1}^{(j_2)}(x)\bigr), \\ \nonumber
\hat T_{j_1,j_2}(x) & := \sum_{(l,k) \in J} \Bigl[ \bigl(\vv_{l,k,0}^{(j_1)}(x) - \tilde \vv_{l,k,0}^{(j_1)}(x)\bigr)\tilde \eta_{l,k,0}^{(j_2)}(x) - \bigl( \vv_{l,k,1}^{(j_1)}(x) - \tilde \vv_{l,k,1}^{(j_1)}(x)\bigr)\tilde \eta_{l,k,1}^{(j_2)}(x)\Bigr].
\end{align}

The summation of $\tilde T_{j_1,j_2}(x)$ and $\hat T_{j_1,j_2}(x)$ in \eqref{sumTj} and below is understood in the sense \eqref{sum}.

Let us estimate the difference
\begin{align} \nonumber
\tilde T_{j_1,j_2,(1)}(x) - \tilde T_{j_1,j_2,(2)}(x) = & \sum_J \bigl( \tilde \vv_{l,k,0,(1)}^{(j_1)}(x) \tilde \eta_{l,k,0,(1)}^{(j_2)}(x) - \tilde \vv_{l,k,0,(2)}^{(j_1)}(x) \tilde \eta_{l,k,0,(2)}^{(j_2)}(x)\bigr) \\ \nonumber = & \sum_J \bigl( \tilde \vv_{l,k,0,(1)}^{(j_1)}(x) - \tilde \vv_{l,k,0,(2)}^{(j_1)}(x) \bigr) \tilde \eta_{l,k,0,(1)}^{(j_2)}(x) \\ \label{diftT} & + \sum_J \tilde \vv_{l,k,0,(2)}^{(j_1)}(x) \bigl( \tilde \eta_{l,k,0,(1)}^{(j_2)}(x) - \tilde \eta_{l,k,0,(2)}^{(j_2)}(x)\bigr).
\end{align}

Note that the functions $\tilde \vv_{l,k,0,(u)}(x)$ and $\tilde \eta_{l,k,0,(u)}(x)$ ($u = 1, 2$) solve the equations $\tilde \ell_n(y) = \la_{l,k,(u)} y$ and $\tilde \ell_n^{\star}(z) = \la_{l,k,(u)} z$, respectively, with the fixed coefficients $\tilde \tau = \{ \tilde \tau_{\nu} \}_{\nu = 0}^{n-2}$. Moreover, due to \eqref{bbound}, there hold
$$
|\la_{l,k,(u)} - \tilde \la_{s,k\pm 1}| \ge \de \quad \text{for all $l,s \ge 1$ and $k$}, \quad u = 1, 2.
$$

Consequently, we obtain the following estimates
analogously to Lemmas~\ref{lem:eta} and~\ref{lem:estvv}:
\begin{gather*}
|\tilde \vv_{l,k,0}^{(j)}(x)| \le C l^j w_{l,k}(x), \quad
|\tilde \vv_{l,k,0,(1)}^{(j)}(x) - \tilde \vv_{l,k,0,(2)}^{(j)}(x)| \le C l^j w_{l,k}(x) \zeta_l, \\
|\tilde \eta_{l,k,0}^{(j)}(x)| \le C l^j w_{l,k}^{-1}(x), \quad
|\tilde \eta_{l,k,0,(1)}^{(j)}(x) -  \tilde \eta_{l,k,0,(2)}^{(j)}(x)| \le C l^j w_{l,k}^{-1}(x) \zeta_l,
\end{gather*}
where $(l,k) \in J$, $j = \overline{0,n-1}$, $x \in [0,1]$, $w_{l,k}(x)$ is defined by \eqref{defw}, $C = C(\Omega,K,\de)$, and 
\begin{equation} \label{defzeta}
\zeta_l := \sum_{k = 1}^{n-1} \Bigl( l^{-(n-1)} |\la_{l,k,(1)} - \la_{l,k,(2)}| + l^{-n} |\be_{l,k,(1)} - \be_{l,k,(2)}|\Bigr).
\end{equation}
Applying these estimates to \eqref{diftT}, we get
\begin{equation} \label{diftT2}
\max_{x \in [0,1]} \bigl| \tilde T_{j_1,j_2,(1)} - \tilde T_{j_1,j_2,(2)}\bigr| \le C(\Omega,K,\de) \sum_{l =1}^{\infty} l^{j_1 + j_2} \zeta_l.
\end{equation}

Recall that $j_1 + j_2 = n - s - 1 \le n - 2$. The definitions \eqref{defZ} and \eqref{defzeta} imply 
\begin{equation} \label{sumzeta}
\sum_{l= 1}^{\infty} l^{n-2} \zeta_l \le C Z.
\end{equation}

Combining \eqref{diftT2} and \eqref{sumzeta}, we deduce
\begin{equation} \label{tTL2}
\bigl\| \tilde T_{j_1,j_2,(1)} - \tilde T_{j_1,j_2,(2)}\bigr\|_{L_2[0,1]} \le C(\Omega,K,\de) Z.
\end{equation}

Next, we need the following lemma.

\begin{lem} \label{lem:reg}
Suppose that $j_1, j_2 \ge 0$ and $j_1 + j_2 = n-1$. Then, there exist constants $\{ a_{l,k,j_1,j_2} \}_{l,k \in J}$ such that the series
$$
\tilde T_{j_1,j_2}^{reg}(x) := \sum_{(l,k) \in J} \sum_{(l,k) \in J} \bigl( \tilde \vv_{l,k,0}^{(j_1)}(x)\tilde \eta_{l,k,0}^{(j_2)}(x) - \tilde \vv_{l,k,1}^{(j_1)}(x) \tilde \eta_{l,k,1}^{(j_2)}(x) - a_{l,k,j_1,j_2}\bigr)
$$
converges in $L_2[0,1]$ and 
$$
\| \tilde T_{j_1,j_2}^{reg} \|_{L_2[0,1]} \le C(\Omega,K,\de) Z.
$$
Moreover, the constants can be chosen so that
\begin{equation} \label{rela}
a_{l,k,j_1,j_2} = -a_{l,k,j_1-1,j_2+1}.
\end{equation}
\end{lem}

Lemma~\ref{lem:reg} is proved similarly to \cite[Lemma~8]{Bond22}, so we omit the details. In short, the Weyl solutions $\tilde \Phi_{k+1}(x,\la)$ and $\tilde \Phi^{\star}_{n-k+1}(x,\la)$ are expanded by the Birkhoff solutions and the asymptotics of Proposition~\ref{prop:Birk} are applied. The proof of \cite[Lemma~8]{Bond22} provides certain formulas for the constants $a_{l,k,j_1,j_2}$. However, for our purposes, the relation \eqref{rela} is sufficient.

Continue the proof of Theorem~\ref{thm:unistab2}.
For definiteness, consider the case $j_1 + j_2 = n-2$, $s = 1$, $T_{j_1,j_2} \in W_2^1[0,1]$. Consider the derivative
$$
\tilde T_{j_1,j_2}'(x) = \tilde T_{j_1+1,j_2}(x) + \tilde T_{j_1,j_2+1}(x).
$$

By Lemma~\ref{lem:reg}, the series $\bigl( \tilde T_{j_1+1,j_2,(1)}(x) - \tilde T_{j_1+1,j_2,(2)}(x)\bigr)$ and $\bigl(\tilde T_{j_1,j_2+1,(1)}(x) - \tilde T_{j_1,j_2+1,(2)}(x)\bigr)$ converge in $L_2[0,1]$ with regularizing constants $a_{l,k,j_1+1,j_2}$ and $a_{l,k,j_1,j_2+1}$, respectively, which satisfy $a_{l,k,j_1+1,j_2} + a_{l,k,j_1,j_2+1}=0$. Hence, the sum $\bigl(\tilde T_{j_1,j_2,(1)}'(x) - \tilde T_{j_1,j_2,(2)}'(x)\bigr)$ converges in $L_2[0,1]$ without regularization and
\begin{equation} \label{diftT3}
\bigl\| \tilde T_{j_1,j_2,(1)}' - \tilde T_{j_1,j_2,(2)}' \bigr\|_{L_2[0,1]} \le C(\Omega,K,\de) Z.
\end{equation}

Combining \eqref{tTL2} and \eqref{diftT3}, we get
\begin{equation} \label{esttT}
\| \tilde T_{j_1,j_2,(1)} - \tilde T_{j_1,j_2,(2)} \|_{W_2^s[0,1]} \le C(\Omega,K,\de) Z
\end{equation}
for $s = 1$. The proof for $s \in \{ 2, 3, \dots, n-1 \}$ is analogous.

Next, one can obtain the similar estimate
\begin{equation} \label{esthT}
\| \hat T_{j_1,j_2,(1)} - \hat T_{j_1,j_2,(2)} \|_{W_2^s[0,1]} \le C(\Omega,K,\de) Z
\end{equation}
under the conditions \eqref{ind}. A detailed derivation of \eqref{esthT} for $n = 2$ is presented in \cite{Bond25}. 

The estimates \eqref{esttT} and \eqref{esthT} together imply \eqref{difT}, which concludes the proof.
\end{proof}

Thus, Theorem~\ref{thm:unistab2} establishes the uniform stability of Inverse Problem~\ref{ip:main} on the set $\mathcal S_{\Omega,K,\de}$ of spectral data. Let us discuss the constraints \eqref{bbound} and \eqref{boundK} forming this set. The condition \eqref{boundK} assumes the existence of the inverse operator $(I - \tilde R(x))^{-1}$ and its uniform bound. In the non-self-adjoint case, conditions of this kind are essential. Even for the second-order Sturm-Liouville operator with complex-valued potential, there are unknown necessary and sufficient conditions on spectral data that guarantee 
the unique solvability 
of the main equation \eqref{maineq}. 
Therefore, the known characterization theorems for the non-self-adjoint case (see, e.g., \cite[Theorem~1.6.3]{FY01}) require the existence of the bounded inverse operator $(I - \tilde R(x))^{-1}$. Moreover, the uniform stability of the inverse Sturm-Liouville problem in the non-self-adjoint case has been obtained in \cite{Bond25} under the uniform bound \eqref{boundK} for the inverse operator. 

Note that the condition \eqref{boundK} is related to a fixed model vector $\tilde \tau$. Furthermore, for $n > 2$, we have to impose the requirement \eqref{bbound} of separation from the eigenvalues of the model vector. 
In fact, for every $\tau \in \mathbf{W}_{simp}$ there exists a model vector satisfying this requirement. However, in Definition~\ref{def:SOe}, the model vector $\tilde \tau$ is fixed, since it is used in construction of the operator $\tilde R(x)$ in the constraint \eqref{boundK}. For the self-adjoint case, we will show in the next sections that the requirement \eqref{boundK} can be omitted, and so \eqref{bbound} can also be removed. But in the non-self-adjoint case the both constraints \eqref{bbound} and \eqref{boundK} are needed.

\section{Self-adjoint case} \label{sec:sa}

This section aims to prove Theorem~\ref{thm:uni} basing on Theorem~\ref{thm:unistab2}.

Suppose that $\Omega, \de > 0$ and $\tilde \tau \in \mathbf{W}_{simp}^+$. Let $\{ \la_{l,k}, \be_{l,k} \}_J$ be any sequence of the set $\mathcal S_{\Omega,\de}^+(\tilde \tau)$ introduced by Definition~\ref{def:SOe+}. 
Under the condition \eqref{bbound}, the operator $\tilde R(x)$ is correctly defined. Moreover, each sequence $\{ \la_{l,k}, \be_{l,k} \}_J \in \mathcal S_{\Omega,\de}^+ \subset \mathcal S^+$ is the spectral data of some $\tau \in \mathbf{W}_{simp}^+$ by Proposition~\ref{prop:sc}. Therefore, there exists a bounded inverse operator $(I - \tilde R(x))^{-1}$ on $m$ for each fixed $x \in [0,1]$. The following lemma shows that these operators are uniformly bounded on $\mathcal S_{\Omega,\de}^+$.

\begin{lem} \label{lem:comp}
For every sequence $\{ \la_{l,k}, \be_{l,k} \}_J$ lying in $\mathcal S^+_{\Omega,\de}$ and satisfying \eqref{bbound}, the corresponding operator $(I - \tilde R(x))^{-1}$ satisfies \eqref{boundK}, where $K = K(\Omega,\de)$.
\end{lem}

\begin{proof}
The lemma is proved by contradiction.
Suppose that there exists a sequence $\{ \la_{l,k,(u)}, \be_{l,k,(u)} \}_J$ $(u \ge 1)$ such that
\begin{equation} \label{contr}
\lim_{u \to \infty} \sup_{0 \le x \le 1} \| (I - \tilde R_{(u)}(x))^{-1} \| = \infty.
\end{equation}
Here and below in this proof, we denote $\| . \| = \| . \|_{m \to m}$.

\smallskip

\textit{Step 1. Passing to the limit}. The data $\la_{l,k,(u)}$ and $\be_{l,k,(u)}$ satisfy the asymptotics \eqref{asymptla} and \eqref{asymptbe}, respectively, where the constants $c_{j,k}$ and $d_{j,k}$ are fixed and the remainder terms satisfy
$$
|\varkappa_{l,k,(u)}| \le |\tilde \varkappa_{l,k}| + \Omega, \quad
|\eta_{l,k,(u)}| \le |\tilde \eta_{l,k}| + C \Omega, \quad (l,k) \in J, \: u \ge 1,
$$
by virtue of \eqref{defXi}. Therefore, we can pass to a subsequence
$\{ \la_{l,k,(u)}, \be_{l,k,(u)} \}_J$ that element-wise converges to a limit $\{ \la_{l,k}, \be_{l,k} \}_J$ as $u \to \infty$. In other words, for each fixed $(l,k) \in J$, there hold
\begin{equation} \label{weaklim}
\lim_{u \to \infty} \la_{l,k,(u)} = \la_{l,k}, \quad \lim_{u \to \infty} \be_{l,k,(u)} = \be_{l,k}.
\end{equation}

The limiting values have the asymptotics
\begin{align*} 
\la_{l,k} & = (-1)^{n-k} \bigl(c_{0,k} l^n + c_{1,k} l^{n-1} + c_{2,k} l^{n-2} + \dots + c_{n-1,k} l + O(1) \bigr), \\
\be_{l,k} & = -n \la_{l,k} \bigl( 1 + d_{1,k} l^{-1} + \dots + d_{n-2,k} l^{-(n-2)} + O(l^{-(n-1)}) \bigr), \quad (l,k) \in J.
\end{align*}

Since the data $\{ \la_{l,k,(u)}, \be_{l,k,(u)} \}_J$ for each fixed $u \ge 1$ fulfill the conditions \eqref{sa}, \eqref{addhyp1}, \eqref{addhyp2}, \eqref{bbound1}, and \eqref{bbound}, so do $\{ \la_{l,k}, \be_{l,k} \}_J$. In particular, bounds \eqref{bbound1} imply (A-1) and (A-2). Consequently, by virtue of \cite[Theorem~1]{Bond24}, the values $\{ \la_{l,k}, \be_{l,k} \}_J$ are the spectral data of some vector $\tau = \{ \tau_{\nu} \}_{\nu = 0}^{n-2}$, $\tau_{\nu} \in W_2^{\nu-1}[0,1]$ for $\nu = \overline{0,n-2}$. Note that the functions $\tau_{\nu}$ may have lower smoothness than $\tau_{\nu,(u)}$. Anyway, the bounds \eqref{bbound} imply \eqref{assump}, so one can construct the operator $\tilde R(x)$ as in Section~\ref{sec:maineq}, using the spectral data $\{ \la_{l,k}, \be_{l,k} \}_J$ and the model vector $\tilde \tau$. Moreover, there exists an inverse operator $(I - \tilde R(x))^{-1}$ on $m$. This operator is uniformly bounded with respect to $x \in [0,1]$: 
\begin{equation} \label{invR}
\sup_{0 \le x \le 1} \| (I - \tilde R(x))^{-1} \|  < \infty.
\end{equation}
since $\tilde R(x)$ is continuous by $x$ (see formulas \eqref{defR1} and \eqref{defR2}). 

\smallskip

\textit{Step 2. Operator sequence}. Let us prove that the sequence $\{ \tilde R_{(u)}(x) \}_1^{\infty}$ converges to $\tilde R(x)$ in the operator norm $\|.\|_{m \to m}$ uniformly by $x \in [0,1]$. 
Denote $V_N := \{ v = (l,k,\eps) \in V \colon l \le N \}$ for $N \ge 1$.
According to \eqref{opR} and \eqref{estop}, we have
\begin{align} \nonumber
\| \tilde R_{(u)}(x) & - \tilde R(x) \| = \sup_{v_0 \in V} \sum_{v \in V} |\tilde R_{v_0,v,(u)}(x) - \tilde R_{v_0,v}(x)| \\ \nonumber & \le \sup_{v_0 \in V} \left( \sum_{v \in V_N} |\tilde R_{v_0,v,(u)}(x) - \tilde R_{v_0,v}(x)| + \sum_{v \in V \setminus V_N} |\tilde R_{v_0,v,(u)}(x)| + \sum_{v \in V \setminus V_N} |\tilde R_{v_0,v}(x)|\right) \\ \label{difRu} & \le
\max \bigl\{ \mathscr M_{N,1}(x), \mathscr M_{N,2}(x)\bigr\} + \mathscr M_{N,3}(x),
\end{align}
where
\begin{align*}
& \mathscr M_{N,1}(x) := \max_{v_0 \in V_{2N}} \sum_{v \in V_N} |\tilde R_{v_0,v,(u)}(x) - \tilde R_{v_0,v}(x)|, \\
& \mathscr M_{N,2}(x) := \sup_{v_0 \in V \setminus V_{2N}} \left( \sum_{v \in V_N} |\tilde R_{v_0,v,(u)}(x)| + \sum_{v \in V_N} |\tilde R_{v_0, v}(x)|\right), \\
& \mathscr M_{N,3}(x) := \sup_{v_0 \in V} \left( \sum_{v \in V \setminus V_N} |\tilde R_{v_0, v,(u)}(x)| + \sum_{v \in V \setminus V_N} |\tilde R_{v_0,v}(x)|\right).
\end{align*}

Due to \eqref{estpsiR}, there holds
\begin{equation*} 
|\tilde R_{v_0, v, (u)}(x)| \le \frac{C\xi_{l,(u)}}{|l-l_0|+1}, \quad 
|\tilde R_{v_0,v}(x)| \le \frac{C \xi_l}{|l-l_0|+1},
\end{equation*}
where $v_0 = (l_0,k_0,\eps_0)$, $v = (l,k,\eps)$, $u \ge 1$, $x \in [0,1]$, and $C = C(\Omega,\de)$ in our case. Moreover, it follows from \eqref{Xixi} that $\xi_{l,(u)} \le \Omega l^{-1}$ for all $l, u \ge 1$, and so $\xi_l \le \Omega l^{-1}$.
Hence
\begin{align*}
\mathscr M_{N,2}(x) & \le C(\Omega,\de) \sup_{l_0 > 2N} \sum_{l = 1}^N \frac{1}{l (|l-l_0|+1)} \le C(\Omega,\de) \left( \sum_{l = N+2}^{2N+1} \frac{1}{l^2} \right)^{1/2} \le \frac{C(\Omega,\de)}{\sqrt N}, \\
\mathscr M_{N,3}(x) & \le C(\Omega,\de) \sup_{l_0 \ge 1} \sum_{l = N+1}^{\infty} \frac{1}{l(|l-l_0|+1)} \le C(\Omega,\de) \left(\sum_{l = N+1}^{\infty} \frac{1}{l^2}\right)^{1/2} \le \frac{C(\Omega,\de)}{\sqrt N}.
\end{align*}

Therefore, for each $\epsilon > 0$, one can choose so large $N$ that $\mathscr M_{N,2}(x) \le \frac{\epsilon}{2}$ and $\mathscr M_{N,3}(x) \le \frac{\epsilon}{2}$. Note that $\mathscr M_{N,1}(x)$ depends only on a finite number of terms $|\tilde R_{v_0,v,(u)}(x) - \tilde R_{v_0,v}(x)|$ for the fixed $N$. Taking \eqref{defR1}, \eqref{defR2}, and \eqref{weaklim} into account, we conclude that $|\tilde R_{v_0,v,(u)}(x) - \tilde R_{v_0,v}(x)| \to 0$ as $u \to \infty$ for fixed $v_0, v$, and so $\mathscr M_{N,1}(x) \to 0$ as $u \to \infty$ for the fixed $N$ uniformly by $x \in [0,1]$. At this step, the continuity of $\tilde R_{v_0,v}(x)$ with respect to $\la_{l,k}$ and $\la_{l_0,k_0}$ is essential, as pointed out in Remark~\ref{rem:modif}.
Thus, one can choose $u_0$ such that $\mathscr M_{N,1}(x) \le \frac{\epsilon}{2}$ for all $u \ge u_0$. Returning to \eqref{difRu}, we conclude that $\| \tilde R_{(u)}(x) - \tilde R(x) \| \to 0$ as $u \to \infty$ uniformly by $x \in [0,1]$. This together with \eqref{invR} contradict to \eqref{contr} and so prove the lemma.
\end{proof}

\begin{remark} \label{rem:sep}
Without loss of generality, one may assume that, for any $\{ \la_{l,k}, \be_{l,k} \}_J$ in $\mathcal S_{\Omega,\de}^+$, the condition \eqref{bbound} holds. Indeed, in view of \eqref{defXi}, the eigenvalues $\{ \la_{l,k} \}_J$ satisfy the asymptotics \eqref{asymptla}, where the coefficients $c_{j,k}$ coincide with the corresponding coefficients $\tilde c_{j,k}$ of the model vector $\tilde \tau$ for $j = \overline{1,n}$ and $\| \{ \varkappa_{l,k} - \tilde \varkappa_{l,k} \} \|_{l_2} \le \Omega$. Consequently, the bounds \eqref{bbound} hold when $l \ge l_*$ or $s \ge l_*$, where $l_*$ depends only on $\tilde \tau$ and $\Omega$. At the same time, the eigenvalues $\la_{l,k}$ for $l < l_*$ lie inside a circle $|\la| \le r$ of radius $r = r(\tilde \tau,\Omega)$. Consequently, relying on Proposition~\ref{prop:sc}, ne can choose another model vector $\tilde{\tilde \tau} \in \mathbf{W}_{simp}^+$ such that
$|\la_{l,k} - \tilde{\tilde\la}_{s,k\pm1}| \ge \de$ for all $l,s\ge1$ and $k$. Indeed, one can put $\tilde{\tilde \la}_{l,k} = \tilde\la_{l,k}$ and $\tilde{\tilde \be}_{l,k} = \tilde \be_{l,k}$ for $l \ge l_*$, choose $\tilde{\tilde \la}_{l,k}$ outside the circle $|\la|\le r$ separated from the other values $\la_{s,k}$, $s \ge l_*$, and choose suitable $\tilde{\tilde \be}_{l,k} \ne 0$ so that the assumptions (A-1), (A-2), \eqref{sa}, \eqref{addhyp1}, and \eqref{addhyp2} are satisfied for $\{ \tilde{\tilde \la}_{l,k}, \tilde{\tilde \be}_{l,k} \}_J$. Then, the corresponding model vector $\tilde{\tilde \tau}$ can be used instead of $\tilde \tau$ for constructing the operator $\tilde R(x)$ as described in Section~\ref{sec:maineq}. We stress that a single model vector $\tilde{\tilde \tau}$ can be chosen for all sequences $\{ \la_{l,k}, \be_{l,k} \}_J \in \mathcal S^+_{\Omega,\de}(\tilde \tau)$. 
\end{remark}

\begin{proof}[Proof of Theorem~\ref{thm:uni}]
According to Lemma~\ref{lem:comp} and Remark~\ref{rem:sep}, there holds $\mathcal S_{\Omega,\de}^+(\tilde \tau) \subset \mathcal S_{\Omega_1,K,\de}(\tilde{\tilde \tau})$ for a suitable $\tilde{\tilde \tau} \in \mathbf{W}^+_{simp}$ and some constants $\Omega_1 = \Omega_1(\Omega,\de)$, $K = K(\Omega,\de)$. Therefore, the conclusion of Theorem~\ref{thm:uni} follows from Theorem~\ref{thm:unistab2}.
\end{proof}

\section{Reconstruction by $(2n-2)$ spectra} \label{sec:sp}

In this section, we study Inverse Problem~\ref{ip:sp} by the eigenvalue set $\{ \la_{l,k}, \mu_{l,k} \}_J$. Following the approach of \cite[Section~2.5.3]{Yur02}, we construct the weight numbers $\{ \be_{l,k} \}_J$ and so reduce Inverse Problem~\ref{ip:sp} to Inverse Problem~\ref{ip:main} by the spectral data. Next, we investigate the properties of the eigenvalues $\{ \mu_{l,k} \}_J$ and prove Theorems~\ref{thm:scsp} and \ref{thm:unisp} basing on the corresponding results for Inverse Problem~\ref{ip:main}.

For $k = \overline{1,n}$, let $\mathcal C_k(x,\la)$ be the solution of equation \eqref{eqv} under the initial conditions
$$
\mathcal C_k^{(j-1)}(0,\la) = \de_{k,j}, \quad j = \overline{1,n}.
$$
Clearly, the functions $\mathcal C_k^{(j-1)}(x,\la)$ are entire in $\la$ for each fixed $x \in [0,1]$, $k,j = \overline{1,n}$. 

The functions $\{ \mathcal C_k(x, \la) \}_{k = 1}^n$ and $\{ \Phi_k(x, \la) \}_{k = 1}^n$ form fundamental systems of solutions of equation \eqref{eqp}. Therefore, the following relations hold:
$$
\Phi_k(x,\la) = \sum_{j = 1}^n m_{j,k}(\la) \mathcal C_j(x, \la), \quad k = \overline{1,n},
$$
where the matrix of the coefficients $M(\la) = [m_{j,k}(\la)]_{j,k = 1}^n$ is called \textit{the Weyl-Yurko matrix}.

The elements of the Weyl-Yurko matrix satisfy the following relations (see \cite[Theorem~2.1.1]{Yur02}):
\begin{equation} \label{relmjk}
m_{j,k}(\la) = \de_{j,k}, \quad j \le k, \qquad
m_{j,k}(\la) = -\frac{\Delta_{j,k}(\la)}{\Delta_{k,k}(\la)}, \quad j > k,
\end{equation}
where 
\begin{equation} \label{defDelta}
\Delta_{k,k}(\la) := \det([\mathcal C_j^{(n-s)}(1,\la)]_{s,j = k+1}^n), \quad k = \overline{1,n-1}, 
\end{equation}
and $\Delta_{j,k}(\la)$ is obtained from $\Delta_{k,k}(\la)$ by replacing $\mathcal C_j$ by $\mathcal C_k$. Obviously, the functions $\Delta_{j,k}(\la)$ ($1 \le k \le j \le n$) are entire in $\la$. 

For each $k \in \{ 1, 2, \dots, n-1 \}$, the eigenvalues $\{ \la_{l,k} \}_{l \ge 1}$ and $\{ \mu_{l,k} \}_{l \ge 1}$ of the corresponding boundary value problems $\mathcal L_k$ \eqref{eqv}, \eqref{bc} and $\mathcal M_k$ \eqref{eqv}, \eqref{bcmu} coincide with the zeros of the functions $\Delta_{k,k}(\la)$ and $\Delta_{k+1,k}(\la)$, respectively (counting with multiplicities).
In view of \eqref{defbe} and \eqref{relmjk}, the functions $M_k(\la) = m_{k,k}(\la)$ are meromorphic in $\la$, their poles belong to the set $\{ \la_{l,k} \}_{l \ge 1}$, and
$$
\be_{l,k} = -\Res_{\la = \la_{l,k}} \frac{\Delta_{k+1,k}(\la)}{\Delta_{k,k}(\la)}, \quad (l,k) \in J.
$$

Under the assumptions (A-1) and (A-2), we have
\begin{equation} \label{beDe}
\be_{l,k} = -\frac{\Delta_{k+1,k}(\la_{l,k})}{\dot\Delta_{k,k}(\la_{l,k})} \ne 0, \quad (l,k) \in J,
\end{equation}
where $\dot \Delta(\la) = \tfrac{d}{d\la}\Delta(\la)$.

The characteristic functions can be constructed as infinite products by using their zeros (see \cite[Theorem~1.1.4]{FY01}):
\begin{align} \label{prod1}
\Delta_{k,k}(\la) & = \Delta_{k,k}^0(0) \prod_{s = 1}^{\infty} \frac{\la_{s,k} - \la}{\la_{s,k}^0}, \\ \label{prod2} \Delta_{k+1,k}(\la) & = \Delta_{k+1,k}^0(0) \prod_{s = 1}^{\infty} \frac{\mu_{s,k} - \la}{\mu_{s,k}^0},
\end{align}
where the values $\Delta_{k,k}^0(0) \ne 0$, $\Delta_{k+1,k}^0(0) \ne 0$, $\la_{s,k}^0$, and $\mu_{s,k}^0$ correspond to the zero vector $\tau$. 

Using \eqref{prod1}, \eqref{prod2}, and \eqref{beDe}, Inverse Problem~\ref{ip:sp} by the eigenvalue set $\{ \la_{l,k}, \mu_{l,k}\}_J$ is reduced to Inverse Problem~\ref{ip:main} by the spectral data $\{ \la_{l,k}, \be_{l,k}\}_J$.

Proceed to proving the main results for Inverse Problem~\ref{ip:sp}.

\begin{proof}[Proof of Lemma~\ref{lem:nscsp}]
It follows from \eqref{beDe} that $\Delta_{k+1,k}(\la_{l,k}) \ne 0$, which readily implies \eqref{sepsp}. For even $n$, the problems $\mathcal L_k$ and $\mathcal M_k$ are adjoint to $\mathcal L_{n-k}$ and $\mathcal M_{n-k}$, respectively. For odd $n$, the similar relations hold for the boundary value problems for equation $i\ell_n(y) = \lambda y$. This implies \eqref{sasp}. 

Proceed to proving the interlacing property~\eqref{inter}. Let $n = 2p$. Then, the problems $\mathcal L_p$ and $\mathcal M_p$ are self-adjoint, so their eigenvalues are real. For definiteness, let $p$ be even. Then $\lambda_{l,p}$ and $\mu_{l,p}$ tend to $+\infty$ as $l \to \infty$, so without loss of generality we assume that $\la_{l,p} < \la_{l+1,p}$ and $\mu_{l,p} \le \mu_{l+1,p}$. 

Introduce the real-valued functionals
\begin{align*}
\mathcal J[y] & := \int_0^1 \left( |y^{(p)}|^2 + \sum_{k = 0}^{p-1} (-1)^k \tau_{2k} |y^{(k)}|^2 + 2 i \sum_{k = 0}^{p-2} (-1)^k \tau_{2k+1} \mbox{Im}\bigl(\overline{y}^{(k)} y^{(k+1)}\bigr)\right)\, dx, \\
\mathcal H[y] & := \int_0^1 |y|^2 dx
\end{align*}
and the domains
\begin{align*}
V_{\la} & := \bigl\{ y \in W_2^p[0,1] \colon y^{(j)}(0) = y^{(j)}(1) = 0, \, j = \overline{0,p-1}\}, \\
V_{\mu} & := \bigl\{ y \in W_2^p[0,1] \colon y^{(j)}(0) = 0, \, j = \overline{0,p-2}, \, y^{(j)}(1) = 0, \, j = \overline{0,p-1}\}.
\end{align*}

According to the minimum property \cite[Chap.~VI, \S~1]{CH89}, 
the eigenvalues of the problems $\mathcal L_p$ and $\mathcal M_p$ solve the corresponding variational problems:
\begin{equation*}
\la_{l,p} = \min_{\substack{y \in V_{\la} \colon y \perp y_s, \\ s = \overline{1,l-1}}} \frac{\mathcal J[y]}{\mathcal H[y]}, \quad
\mu_{l,p} = \min_{\substack{y \in V_{\mu} \colon y \perp y_s, \\ s = \overline{1,l-1}}} \frac{\mathcal J[y]}{\mathcal H[y]},
\end{equation*}
where $\{ y_s \}_{s \ge 1}$ and $\{ z_s \}_{s \ge 1}$ are the eigenfunctions of $\mathcal L_p$ and $\mathcal M_p$, respectively. Furthermore, due to the maximum-minimun property of eigenvalues \cite[Chap.~VI, \S~1]{CH89}, there holds
$$
\la_{l,p} = \max_{\{ \mathscr L_s\}_1^{l-1}} \min_{\substack{y \in V_{\la} \colon \mathscr L_s(y) = 0, \\ s = \overline{1,l-1}}} \frac{\mathcal J[y]}{\mathcal H[y]}, \quad 
\mu_{l,p} = \max_{\{ \mathscr L_s\}_1^{l-1}} \min_{\substack{y \in V_{\mu} \colon \mathscr L_s(y) = 0, \\ s = \overline{1,l-1}}} \frac{\mathcal J[y]}{\mathcal H[y]},
$$
where the maximum is taken over all sets of $(l-1)$ continuous linear functionals $\mathscr L_s$ on $W_2^p[0,1]$. Since $V_{\la} \subset V_{\mu}$, we immediately get $\mu_{l,p} \le \la_{l,p}$. On the other hand, there holds
$$
\la_{l,p} = \max_{\{ \mathscr L_s\}_1^{l-1}} \min_{\substack{y \in V_{\mu} \colon y^{(p-1)}(0) = 0, \\ \mathscr L_s(y) = 0, \, s = \overline{1,l-1}}} \frac{\mathcal J[y]}{\mathcal H[y]} \le
\max_{\{ \mathscr L_s\}_1^l} \min_{\substack{y \in V_{\mu} \colon \mathscr L_s(y) = 0, \\ s = \overline{1,l}}} \frac{\mathcal J[y]}{\mathcal H[y]} = \mu_{l,p+1},
$$
where for $\la_{l,p}$ we have the particular linear functional $\mathscr L_l(y) := y^{(p-1)}(0)$. Thus $\mu_{l,p} \le \la_{l,p} \le \mu_{l+1,p}$, which together with \eqref{sepsp} imply \eqref{inter} for even $p$. The case of odd $p$ is studied analogously.
\end{proof}

\begin{proof}[Proof of Theorems~\ref{thm:scsp} and~\ref{thm:unisp}]
Suppose that $\tilde \tau \in \mathbf{W}_{simp}^+$ and $\{ \la_{l,k}, \mu_{l,k}\}_J \in \mathcal E^+(\tilde \tau)$. Let the functions $\Delta_{k,k}(\la)$, $\Delta_{k+1,k}(\lambda)$ and the numbers $\{ \be_{l,k} \}_J$ be constructed by using \eqref{prod1}, \eqref{prod2}, and \eqref{beDe}, respectively. In order to prove Theorem~\ref{thm:scsp}, it is sufficient to show that $\{ \la_{l,k}, \be_{l,k} \}_J \in \mathcal S^+(\tilde \tau)$. For proving Theorem~\ref{thm:unisp}, we have to show additionally that, if $\{ \la_{l,k}, \mu_{l,k} \}_J \in \mathcal E^+_{\Omega,\de}$, then $\{ \la_{l,k}, \be_{l,k}\}_J \in \mathcal S^+_{\Omega_1, \eps_1}$, where $\Omega_1$ and $\eps_1$ depend only on $\Omega$ and $\eps$, and $Z \le C(\Omega,\de) X$. Then, applying Proposition~\ref{prop:sc} and Theorem~\ref{thm:uni}, one easily concludes the proof.

The most nontrivial part of the proof is to show that 
$$
\sum_{l = 1}^{\infty} \Bigg(\sum_{k = 1}^{n-1} l^{-1}|\be_{l,k} - \tilde \be_{l,k}|\Bigg)^2 \le C(\Omega,\de).
$$

Throughout this proof, we denote by $\{ \varkappa_s \}$ various $l_2$-sequences, whose norms are bounded by $C(\Omega,\de)$ in the case $\{\la_{l,k}, \mu_{l,k}\}_J \in \mathcal E^+_{\Omega,\de}$.

Using \eqref{prod2}, we obtain
\begin{equation} \label{Dkk1}
\Delta_{k+1,k}(\la_{l,k}) = \tilde \Delta_{k+1,k}(\la_{l,k}) \prod_{s = 1}^{\infty} (1 + a_{s,l,k}), \quad a_{s,l,k} := \frac{\mu_{s,k} - \tilde \mu_{s,k}}{\tilde \mu_{s,k} - \la_{l,k}}.
\end{equation}
Recall that $\la_{l,k}$ satisfy \eqref{asymptla} and $\mu_{l,k}$, the similar asymptotics
$$
\mu_{l,k} = (-1)^{n-k} \bigl(f_{0,k} l^n + f_{1,k} l^{n-1} + f_{2,k} l^{n-2} + \dots + f_{l,k}  + \kappa_{l,k} \bigr), \quad \{ \kappa_{l,k}\} \in l_2,
$$
where $f_{0,k} = c_{0,k}$ \eqref{c0k} and $f_{1,k} \ne f_{0,k}$. Moreover, $\{ \ln (l + 1) (\kappa_{l,k} - \tilde \kappa_{l,k})\} \in l_2$ in view of \eqref{defTheta}. Consequently, we get
\begin{equation} \label{esta}
|a_{s,l,k}| \le \frac{\varkappa_l}{l^{n-1} \ln (l + 1) (|l-s|+1)}.
\end{equation}
In particular, for all $l \ge N$, $N = N(\Omega,\eps)$, there holds $|a_{s,l,k}| \le \tfrac{1}{2}$, so
$$
\left| \ln \prod_{s = 1}^{\infty} (1 + a_{s,l,k}) \right| \le \sum_{s = 1}^{\infty} |\ln(1 + a_{s,l,k})| \le C \sum_{s = 1}^{\infty} |a_{s,l,k}|.
$$
Applying \eqref{esta} and Lemma~\ref{lem:bound}, we conclude that the sequence $\{ z_l \}_{l \ge N}$, $z_l := l^{n-1} \sum\limits_{s = 1}^{\infty} |a_{s,l,k}|$, belongs to $l_2$ and $\| \{ z_l \}_{l \ge N} \|_{l_2} \le C(\Omega, \de)$. Consequently, we get
\begin{equation} \label{aprod}
\prod_{s = 1}^{\infty} (1 + a_{s,l,k}) = 1 + \frac{\varkappa_l}{l^{n-1}}, \quad l \ge N.
\end{equation}

Expanding the characteristic function $\tilde\Delta_{k+1,k}(\la)$ in terms of the Birkhoff solutions  (see Proposition~\ref{prop:Birk}) and using the asymptotics \eqref{asymptla} for $\la_{l,k}$ and $\tilde \la_{l,k}$, we obtain
\begin{equation} \label{aD1}
\tilde \Delta_{k+1,k}(\la_{l,k}) = \tilde \Delta_{k+1,k}(\tilde \la_{l,k}) \left( 1 + \frac{\varkappa_l}{l^{n-1}} \right), \quad l \ge N.
\end{equation}
Combining \eqref{Dkk1}, \eqref{aprod}, and \eqref{aD1}, we conclude that 
\begin{equation} \label{aD2}
\Delta_{k+1,k}(\la_{l,k}) = \tilde \Delta_{k+1,k}(\tilde \la_{l,k}) \left( 1 + \frac{\varkappa_l}{l^{n-1}} \right), \quad l \ge N.
\end{equation}
Furthermore, we derive the relation
\begin{equation} \label{aD3}
\dot{\Delta}_{k,k}(\la_{l,k}) = \dot{\tilde \Delta}_{k,k}(\tilde \la_{l,k}) \left( 1 + \frac{\varkappa_l}{l^{n-1}} \right), \quad l \ge N.
\end{equation}
For this purpose, we observe that $\{ \la_{l,k}, \tilde \be_{l,k}\}_J \in \mathcal S^+$, so $\{ \la_{l,k}\}_{l \ge 1}$ are the eigenvalues for some $\check \tau \in \mathbf{W}^+_{simp}$ and $\Delta_{k,k}(\la)$ is the corresonding characteristic function. Expanding $\Delta_{k,k}(\la)$ and $\tilde \Delta_{k,k}(\la)$ in terms of the Birkhoff solutions of equation \eqref{eqv} with $\check \tau$ and $\tilde \tau$, respectively, we derive their asymptotics and obtain \eqref{aD3}.

The relations \eqref{beDe}, \eqref{aD2}, and \eqref{aD3} together imply
$$
\be_{l,k} = \tilde \be_{l,k} \left( 1 + \frac{\varkappa_l}{l^{n-1}} \right), \quad n \ge N.
$$
Taking the asymptotics \eqref{asymptbe} for $\tilde \be_{l,k}$ into account, we conclude that
$$
\sum_{l = N}^{\infty} \Bigg(\sum_{k = 1}^{n-1} l^{-1}|\be_{l,k} - \tilde \be_{l,k}|\Bigg)^2 \le C(\Omega,\de), \quad N = N(\Omega).
$$
The estimate $|\be_{l,k}| \le C(\Omega,\de)$ for $l < N$, as well as the other conditions of Definitions~\ref{def:S+} and~\ref{def:SOe+}, are obtained by using the relations \eqref{prod1}, \eqref{prod2}, and \eqref{beDe} together with the corresponding conditions of Definitions~\ref{def:E+} and~\ref{def:EOde+}, which concludes the proof.   
\end{proof}

\begin{remark} \label{rem:nosep}
In fact, Proposition~\ref{prop:sc} and Theorem~\ref{thm:uni} are valid for any $\tilde \tau \in \mathbf{W}^+$. Indeed, let the spectral data of $\tilde \tau$ do not satisfy (A-1) and (A-2). In view of the asymptotics \eqref{asymptla}, the equalities $\la_{l,k} = \la_{s,k}$ and $\la_{l,k} = \la_{s,k+1}$ are possible only for a finite number of eigenvalues. Therefore, one can achieve (A-1) and (A-2) by a finite perturbation of the spectral data. Then, by using the constructive method of \cite{Bond22}, one can obtain a new vector $\tilde{\tilde \tau} \in \mathbf{W}_{simp}^+$ with the same coefficients in the spectral data asymptotics \eqref{asymptla} and \eqref{asymptbe} as $\tilde \tau$ has.
\end{remark}

\section{Case $n = 3$} \label{sec:3}

For $n = 3$, equation \eqref{eqv} takes the form
\begin{equation} \label{eqv3}
y''' + (\tau_1(x) y)' + \tau_1(x) y' + \tau_0(x) y = \la y, \quad x \in (0,1),
\end{equation}
where $\tau = \{ \tau_0, \tau_1 \} \in \mathbf{W}$ means $\tau_0 \in L_2[0,1]$, $\tau_1 \in W_2^1[0,1]$. The spectral data $\{ \la_{l,k}, \be_{l,k} \}_J$ correspond the the eigenvalue problems $\mathcal L_1$ and $\mathcal L_2$ with the boundary conditions
\begin{align} \label{bc31}
& \mathcal L_1 \colon \quad y(0) = 0, \quad y(1) = y'(1) = 0, \\ \nonumber
& \mathcal L_2 \colon \quad y(0) = y'(0) = 0, \quad y(1) = 0,
\end{align}
and $\{ \mu_{l,k}\}_J$ are the eigenvalues of the corresponding problems $\mathcal M_1$ and $\mathcal M_2$ with the boundary conditions
\begin{align*} 
& \mathcal M_1 \colon \quad y'(0) = 0, \quad y(1) = y'(1) = 0, \\ 
& \mathcal M_2 \colon \quad y(0) = y''(0) = 0, \quad y(1) = 0,
\end{align*}

\begin{thm} \label{thm:asymptsd3}
The spectral data of $\tau = \{ \tau_0, \tau_1 \} \in \mathbf{W}$ for $n = 3$ satisfy the asymptotic relations
\begin{align} \label{asymptla31}
\la_{l,1} = & \bigl(\rho_l^+\bigr)^3 - 2\theta \rho_l^+ + (\theta_0 + \theta_1 + \sigma) + \varkappa_{l,1}, \\ \label{asymptla32}
\la_{l,2} = - & \Bigl( \bigl(\rho_l^+\bigr)^3 - 2\theta \rho_l^+ + (\theta_0 + \theta_1 - \sigma) + \varkappa_{l,2} \Bigr), \\ \label{asymptmu31}
\mu_{l,1} = & \bigl(\rho_l^-\bigr)^3 - 2\theta \rho_l^- + (\sigma - \theta_0 + \theta_1) + \kappa_{l,1}, \\ \label{asymptmu32}
\mu_{l,2} = - & \Bigl( \bigl(\rho_l^-\bigr)^3 - 2\theta \rho_l^- - (\sigma + \theta_0 - \theta_1) + \kappa_{l,2} \Bigr),
\\ \label{asymptbe312}
\be_{l,k} = & -3\la_{l,k} \left( 1 + \frac{\theta + \theta_0}{2 \pi^2 l^2} + \frac{\eta_{l,k}}{l^2}\right), \quad l \ge 1, \: k = 1, 2,
\end{align}
where $\rho_l^{\pm} = \frac{2 \pi}{\sqrt 3} \bigl( l \pm \frac{1}{6}\bigr)$, $\{ \varkappa_{l,k} \}, \, \{ \kappa_{l,k}\}, \, \{ \eta_{l,k} \}, \in l_2$, and
\begin{equation} \label{const3}
\theta = \int_0^1 \tau_1(x) \, dx, \quad \theta_0 = \tau_1(0), \quad
\theta_1 = \tau_1(1), \quad \sigma = \int_0^1 \tau_0(x) \, dx.
\end{equation}
\end{thm}

The proof of Theorem~\ref{thm:asymptsd3} is provided in Appendix~\ref{app:calc3}.

Suppose that $\tau \in \mathbf{W}^+$, that is,
$\tau_1 \in W_2^1[0,1]$ is real-valued and $\tau_0 \in L_2[0,1]$ is purely imaginary. In view of \eqref{sa}, the spectral data $\{ \la_l, \be_l \}_{l \ge 1} := \{ \la_{l,1}, \be_{l,1}\}_{l \ge 1}$ are sufficient for the unique reconstruction of $\tau \in \mathbf{W}_{simp}^+$. Proposition~\ref{prop:sc} implies the following sufficient conditions for the existence of the inverse problem solution. 

\begin{thm} \label{thm:char3}
Let a sequence $\{ \la_l, \be_l \}_{l \ge 1}$ satisfy the conditions $\la_l \ne \la_s$ $(l \ne s)$, $\mbox{Re} \, \la_l > 0$, $\be_l \ne 0$ for all $l \ge 1$, and the asymptotics
\begin{align} \label{asymptla3}
& \la_l = \bigl( \rho_l^+ \bigr)^3 - 2 \theta \rho_l^+ + (\theta_0 + \theta_1 + \sigma) + \varkappa_l, \quad \rho_l^+ := \tfrac{2 \pi}{\sqrt 3} \bigl( l + \tfrac{1}{6} \bigr), \quad \{ \varkappa_l \} \in l_2, \\ \label{asymptbe3}
& \be_l = -3\la_l \left( 1 + \frac{\theta + \theta_0}{2 \pi^2 l^2} + \frac{\eta_l}{l^2} \right), \quad \{ \eta_l \}\in l_2,
\end{align}
where $\theta$, $\theta_1$, and $\theta_0$ are real constants and $\sigma$ is a purely imaginary constant. Then the numbers $\{ \la_l, \be_l \}_{l \ge 1}$ are the spectral data of a unique vector $\tau = \{ \tau_0, \tau_1 \} \in \mathbf{W}_{simp}^+$. Moreover, there holds~\eqref{const3}.
\end{thm}

\begin{proof}
Choose a model vector $\tilde \tau = \{ \tilde \tau_0, \tilde \tau_1 \} \in \mathbf{W}^+$ such that $\tilde \theta = \theta$, $\tilde \theta_0 = \theta_0$, $\tilde \theta_1 = \theta_1$, and $\tilde \sigma = \sigma$ according to \eqref{const3}. For instance, put
\begin{equation} \label{mod3}
\tilde \tau_1(x) := (3\theta_1 + 3\theta_0 - 6\theta)x^2 + (6\theta - 4 \theta_0 - 2\theta_1)x + \theta_0, \quad \tilde \tau_0(x) := \sigma.
\end{equation}
Then, in view of Theorem~\ref{thm:asymptsd3}, the data $\la_{l,1} := \la_l$, $\be_{l,1} := \be_l$, $\la_{l,2} := -\overline{\la_l}$, $\be_{l,2} := -\overline{\be_l}$ belong to $\mathcal S^+(\tilde \tau)$. Proposition~\ref{prop:sc} and Remark~\ref{rem:nosep} conclude the proof.
\end{proof}

Theorem~\ref{thm:char3} is consistent with \cite[Theorem~2.3]{Bond23-res} for the class $\tau_0 \in W_2^{-1}[0,1]$, $\tau_1 \in L_2[0,1]$.

Analogously, Theorem~\ref{thm:scsp} implies the following sufficient conditions for the unique reconstruction of $\tau = \{ \tau_0, \tau_1\} \in \mathbf{W}^+_{simp}$ from the two spectra $\{ \la_l, \mu_l \}_{l \ge 1} := \{ \la_{l,1}, \mu_{l,1}\}_{l \ge 1}$.

\begin{thm} \label{thm:char3sp}
Let a sequence $\{ \la_l, \mu_l \}_{l \ge 1}$ satisfy the conditions $\la_l \ne \la_s$ $(l \ne s)$, $\la_l \ne \mu_s$, $\mbox{Re} \, \la_l > 0$ for all $l,s \ge 1$, the asymptotics \eqref{asymptla3} and
\begin{equation} \label{asymptmu3}
\mu_l = \bigl(\rho_l^-\bigr)^3 - 2\theta \rho_l^- + (\sigma - \theta_0 + \theta_1) + \kappa_l, \quad \rho_l^- := \tfrac{2\pi}{\sqrt 3} \bigl( l - \tfrac{1}{6}\bigr), \quad 
\{\kappa_l\ln (l+1)\} \in l_2,
\end{equation}
where $\theta$, $\theta_0$, and $\theta_1$ are real constants and $\sigma$ is a purely imaginary constant. Then $\{ \la_l, \mu_l \}_{l \ge 1}$ are the eigenvalues of the problems $\mathcal L_1$ and $\mathcal M_1$ for a unique vector $\tau = \{ \tau_0, \tau_1\} \in \mathbf{W}_{simp}^+$. Moreover, there holds \eqref{const3}.
\end{thm}

Similarly, Theorems~\ref{thm:asymptsd3}, \ref{thm:uni}, and~\ref{thm:unisp} imply the following theorems on the uniform stability of Inverse Problem~\ref{ip:main} in the case $n = 3$.

\begin{thm} \label{thm:uni3}
Let $\Omega > 0$ and $\de > 0$ be fixed, and let each of two sequences $\{ \la_{l,(1)}, \be_{l,(1)} \}_{l \ge 1}$ and $\{ \la_{l,(2)}, \be_{l,(2)} \}_{l \ge 1}$ satisfy the conditions
$$
|\la_l - \la_s| \ge \de \: (l \ne s), \quad |\mbox{Re} \, \la_l| \ge \de, \quad |\be_l| \ge \de, \quad l \ge 1,
$$
and the asymptotics \eqref{asymptla3} and \eqref{asymptbe3} with the same real constants $\theta$, $\theta_1$, $\theta_0$, purely imaginary $\sigma$ and such remainder terms that
$$
\| \{ \varkappa_l \}_{l \ge 1} \|_{l_2} \le \Omega, \quad \| \{ \eta_l \}_{l \ge 1} \|_{l_2} \le \Omega.
$$
Then
$$
\| \tau_{0,(1)} - \tau_{0,(2)} \|_{L_2[0,1]} \le C(\Omega,\de) Z_1, \quad
\| \tau_{1,(1)} - \tau_{1,(2)} \|_{W_2^1[0,1]} \le C(\Omega, \de) Z_1,
$$
where 
$$
Z_1 := \left( \sum_{l = 1}^{\infty} \bigl( |\varkappa_{l,(1)} - \varkappa_{l,(2)}| + |\eta_{l,(1)} - \eta_{l,(2)}| \bigr)^2 \right)^{1/2}.
$$
\end{thm}

\begin{thm} \label{thm:uni3sp}
Let $\Omega > 0$ and $\de > 0$ be fixed, and let each of two sequences $\{ \la_{l,(1)}, \mu_{l,(1)} \}_{l \ge 1}$ and $\{ \la_{l,(2)}, \mu_{l,(2)} \}_{l \ge 1}$ satisfy the conditions
\begin{gather*}
|\la_l - \la_s| \ge \de \: (l \ne s), \quad |\mbox{Re} \, \la_l| \ge \de, \quad
|\la_l - \mu_s| \ge \de, \quad l,s \ge 1,
\end{gather*}
and the asymptotics \eqref{asymptla3} and \eqref{asymptmu3} with the same real constants $\theta$, $\theta_1$, $\theta_0$, purely imaginary $\sigma$ and such remainder terms that
$$
\| \{ \varkappa_l \}_{l \ge 1} \|_{l_2} \le \Omega, \quad \| \{ \kappa_l \ln(l + 1) \}_{l \ge 1} \|_{l_2} \le \Omega.
$$
Then
$$
\| \tau_{0,(1)} - \tau_{0,(2)} \|_{L_2[0,1]} \le C(\Omega,\de) X_1, \quad
\| \tau_{1,(1)} - \tau_{1,(2)} \|_{W_2^1[0,1]} \le C(\Omega, \de) X_1,
$$
where 
$$
X_1 := \left( \sum_{l = 1}^{\infty} \bigl( |\varkappa_{l,(1)} - \varkappa_{l,(2)}| + \ln(l+1)|\kappa_{l,(1)} - \kappa_{l,(2)}| \bigr)^2 \right)^{1/2}.
$$
\end{thm}

In a special case, we obtain the following lemma on additional properties of the spectral data $\{ \la_l, \be_l \}_{l \ge 1}$.

\begin{lem} \label{lem:30}
Suppose that $\tau_1$ is a real-valued function of $W_2^1[0,1]$ and $\tau_0 \equiv 0$. Then, for each $l \ge 1$, there exists an index $p(l)$ such that $\la_{p(l)} = \overline{\lambda_l}$ and $\be_{p(l)} = \overline{\be_l}$. Furthermore, the spectrum $\{ \mu_l \}_{l \ge 1}$ is symmetric with respect to the real line counting with multiplicities. 
In particular, for all sufficiently large $l$, the values $\la_l$, $\be_l$, and $\mu_l$ are real.
\end{lem}

\begin{proof}
Complex conjugation of equation \eqref{eqv3} and the formula \eqref{defbe} for the weight numbers, taking the asymptotics \eqref{asymptla3} and \eqref{asymptmu3} into account.
\end{proof}

\begin{thm} \label{thm:30}
Let a sequence $\{ \la_l, \be_l \}_{l \ge 1}$ satisfies the hypothesis of Theorem~\ref{thm:char3} with $\sigma = 0$ and the corresponding conclusions of Lemma~\ref{lem:30}. Then $\{ \la_l, \be_l \}_{l \ge 1}$ are the spectral data of a unique vector $\tau = \{ 0, \tau_1 \} \in \mathbf{W}^+_{simp}$ (i.e. $\tau_0 \equiv 0$).
\end{thm}

\begin{proof}
Theorem~\ref{thm:char3} implies the existence of a vector $\tau = \{ \tau_0, \tau_1 \} \in \mathbf{W}^+_{simp}$, whose  spectral data are the given $\{ \la_l, \be_l \}_{l \ge 1}$. One can easily check that the spectral data of $\overline{\tau} = \{\overline{\tau}_0, \tau_1 \}$ also equal $\{ \la_l, \be_l \}_{l \ge 1}$. By virtue of the uniqueness and $\overline{\tau}_0 = -\tau_0$, we conclude that $\tau_0 \equiv 0$.
\end{proof}

Analogously, we get the sufficient conditions on the two spectra:

\begin{thm} 
Let a sequence $\{ \la_l, \mu_l \}_{l \ge 1}$ satisfies the hypothesis of Theorem~\ref{thm:char3sp} with $\sigma = 0$ and the corresponding conclusions of Lemma~\ref{lem:30}. Then $\{ \la_l, \mu_l \}_{l \ge 1}$ are the eigenvalues of the respective problems $\mathcal L_1$ and $\mathcal M_1$ for a unique vector $\tau = \{ 0, \tau_1 \} \in \mathbf{W}^+_{simp}$.
\end{thm}

The uniform stability of the inverse problems in the special case $\tau_0 \equiv 0$ holds due to Theorems~\ref{thm:uni3} and~\ref{thm:uni3sp}.

\section{Case $n = 4$} \label{sec:4}

For $n = 4$, we have the equation
\begin{equation} \label{eqv4}
y^{(4)} + (\tau_2(x)y')' + (\tau_1(x)y)' + \tau_1(x)y' + \tau_0(x)y = \la y, \quad x \in (0,1),
\end{equation}
where $\tau = \{ \tau_0, \tau_1, \tau_2 \} \in \mathbf{W}^+$ means that the functions $\tau_0 \in L_2[0,1]$ and $\tau_2 \in W_2^2[0,1]$ are real-valued and $\tau_1 \in W_2^1[0,1]$ is purely imaginary. For the recovery of $\tau_0$, $\tau_1$, and $\tau_2$, we use the eigenvalues of the boundary value problems $\mathcal L_k$ ($k = 1, 2, 3$) for equation \eqref{eqv4} with the boundary conditions
\begin{align} \label{bc1}
\mathcal L_1 \colon & \quad y(0) = 0, \quad y(1) = y'(1) = y''(1) = 0, \\ \label{bc2}
\mathcal L_2 \colon & \quad y(0) = y'(0) = 0, \quad y(1) = y'(1) = 0, \\ \nonumber
\mathcal L_3 \colon & \quad y(0) = y'(0) = y''(0) = 0, \quad y(1) = 0.
\end{align}

Note that $\mathcal L_2 = \mathcal L_2^*$ and $\mathcal L_1 = \mathcal L_3^*$.

The following theorem provides the spectral data characterization for $\tau \in \mathbf{W}_{simp}^+$ in the case $n = 4$.

\begin{thm} \label{thm:char4}
For a sequence $\{ \la_{l,k}, \be_{l,k} \}_{l \ge 1, \, k = 1, 2, 3}$ to be the spectral data of $\tau = \{ \tau_0, \tau_1, \tau_2 \} \in \mathbf{W}_{simp}^+$, it is necessary and sufficient to fulfill the conditions (A-1), (A-2), \eqref{sa}, $\be_{l,1} \ne 0$, $\be_{l,2} < 0$, and the asymptotics
\begin{multline} \label{asymptla41}
\la_{l,2\pm1} = -\Biggl( \biggl( \sqrt 2 \pi l + \frac{\pi}{2 \sqrt 2}\biggr)^4 - \theta \biggl( \sqrt 2 \pi l + \frac{\pi}{2 \sqrt 2}\biggr)^2 + \biggl( \frac{\theta_0 + \theta_1}{\sqrt 2} \mp 2 \sqrt 2 \sigma \biggr)\biggl( \sqrt 2 \pi l + \frac{\pi}{2 \sqrt 2}\biggr) \\ + \frac{\theta_1' - \theta_0'}{4} \pm 2 (\sigma_0 + \sigma_1) + \phi + \frac{\theta^2}{8} + \varkappa_{l,2\pm1} \Biggr),
\end{multline}
\vspace*{-0.8cm}
\begin{multline} \label{asymptla42}
\la_{l,2} = \biggl( \pi l + \frac{\pi}{2}\biggr)^4 - \theta \biggl( \pi l + \frac{\pi}{2}\biggr)^2 + (\theta_0 + \theta_1) \biggl( \pi l + \frac{\pi}{2}\biggr) - \frac{3}{4}(\theta_1' - \theta_0') - \phi + \frac{\theta^2}{8} + \varkappa_{l,2},   
\end{multline}
\vspace{-0.5cm}
\begin{align} \label{asymptbe41}
& \be_{l,2\pm1} = -4 \la_{l,2\pm 1} \left(1 + \frac{\theta + \theta_0}{8\bigl( \pi l + \frac{\pi}{4}\bigr)^2} - \frac{2(\theta_0 + \theta_1) - \theta_0' \mp 4 (2 \sigma + \sigma_0)}{16 (\pi l)^3} + \frac{\eta_{l,2\pm1}}{l^3} \right), \\ \label{asymptbe42}
& \be_{l,2} = -4 \la_{l,2} \left(1 + \frac{\theta + \theta_0}{4\bigl( \pi l+ \frac{\pi}{2} \bigr)^2} + \frac{\theta_0' - 2(\theta_0 + \theta_1)}{4 (\pi l)^3} + \frac{\eta_{l,2}}{l^3} \right),
\end{align}
where $\{ \varkappa_{l,k} \}, \, \{ \eta_{l,k} \} \in l_2$ and
\begin{equation} \label{const4}
\def\arraystretch{2}
\left.
\begin{array}{c}
\theta = \displaystyle\int_0^1 \tau_2(x) \, dx, \quad \theta_0 = \tau_2(0), \quad \theta_1 = \tau_2(1), \quad \theta_0' = \tau_2'(0), \quad \theta_1' = \tau_2'(1), \\
\sigma = \displaystyle\int_0^1 \tau_1(x) \, dx, \quad \sigma_0 = \tau_1(0), \quad \sigma_1 = \tau_1(1), \quad \phi = \dfrac{1}{8} \displaystyle\int_0^1 \tau_2^2(x) \, dx - \displaystyle\int_0^1 \tau_0(x) \, dx.
\end{array} \: \right\}
\end{equation}
\end{thm}

In the sufficiency part of Theorem~\ref{thm:char4}, 
the constants $\theta$, $\theta_0$, $\theta_1$, $\theta_0'$, $\theta_1'$, and $\phi$ can be arbitrary real numbers and the constants $\sigma$, $\sigma_0$, and $\sigma_1$, arbitrary purely imaginary complex numbers. Then the relations \eqref{const4} are a part of the conclusion of the theorem.

\begin{proof}[Proof of Theorem~\ref{thm:char4}]
\textit{Necessity}. The asymptotic formulas \eqref{asymptla41}--\eqref{asymptbe42} are proved in Appendix~\ref{app:calc4} for every $\tau \in \mathbf{W}$. The other properties have been obtained in \cite{Bond24}.

\smallskip

\textit{Sufficiency}. Choose a model problem $\tilde \tau = \{ \tilde \tau_0, \tilde \tau_1, \tilde \tau_2 \} \in \mathbf{W}^+$ such that $(\tilde\theta, \tilde\theta_0, \tilde\theta_1, \tilde\theta_0', \tilde\theta_1', \tilde\sigma, \tilde\sigma_0, \tilde\sigma_1, \tilde\phi) =(\theta, \theta_0, \theta_1, \theta_0', \theta_1', \sigma, \sigma_0, \sigma_1, \phi)$ according to \eqref{const4}. For instance, put
\begin{align*}
\tilde\tau_2(x) := & \bigl(30 \theta - 15\theta_0 - 15 \theta_1 - \tfrac{5}{2}\theta_0' + \tfrac{5}{2} \theta_1' \bigr) x^4 
+ \bigl( -60 \theta + 32 \theta_0 + 28 \theta_1 + 6 \theta_0' - 4 \theta_1'\bigr) x^3 \\
& + \bigl( 30 \theta - 18 \theta_0 - 12 \theta_1 - \tfrac{9}{2} \theta_0' + \tfrac{3}{2} \theta_1'\bigr) x^2 + \theta_0' x + \theta_0, \\
\tilde\tau_1(x) := & (3 \sigma_0 + 3 \sigma_1 - 6\sigma) x^2 + (6 \sigma - 4 \sigma_0 - 2 \sigma_1) x + \sigma_0, \\
\tilde \tau_0(x) := & \frac{1}{8} \int_0^1 \tilde \tau_2^2(x) \, dx - \phi.
\end{align*}
Then the conditions of Theorem~\ref{thm:char4} imply $\{ \la_{l,k}, \be_{l,k} \} \in \mathcal S^+(\tilde \tau)$. Proposition~\ref{prop:sc} and Remark~\ref{rem:nosep} conclude the proof.
\end{proof}

Theorems~\ref{thm:uni} and~\ref{thm:char4} together imply the uniform stability of Inverse Problem~\ref{ip:main}.
For clarity, we formulate the corresponding theorem by using the spectral data $\{ \la_{l,k}, \be_{l,k} \}$ only for $k = 1, 2$, since $\la_{l,3}$ and $\be_{l,3}$ are uniquely specified by \eqref{sa}.

\begin{thm} \label{thm:uni4}
Let $\Omega > 0$ and $\de > 0$ be fixed, and let each of two sequences $\{ \la_{l,k,(1)}, \be_{l,k,(1)} \}_{l \ge 1, \, k = 1, 2}$ and
$\{ \la_{l,k,(2)}, \be_{l,k,(2)} \}_{l \ge 1, \, k = 1, 2}$ satisfy the conditions
\begin{gather*}
|\la_{l,k_1} - \la_{s,k_2}| \ge \de, \quad (l,k_1) \ne (s,k_2), \quad
|\be_{l,1}|\ge \de, \quad
l,s\ge 1, \: k_1,k_2 \in \{ 1, 2 \}, \\
\la_{l,2}, \, \be_{l,2} \in \mathbb R, \quad
-\be_{l,2} \ge \de, \quad l \ge 1,
\end{gather*}
 and the asymptotics of Theorem~\ref{thm:char4} with the same real constants $\theta$, $\theta_0$, $\theta_1$, $\theta_0'$, $\theta_1'$, $\phi$, imaginary constants $\sigma$, $\sigma_0$, $\sigma_1$, and with remainders satisfying 
$$
\| \{ \varkappa_{l,k} \}_{l \ge 1} \| \le \Omega, \quad
\| \{ \eta_{l,k} \}_{l \ge 1} \| \le \Omega, \quad k = 1, 2.
$$
Then the estimates \eqref{esttau} hold for $\nu = 0, 1, 2$.
\end{thm}

Consider the special case $\tau_1 \equiv 0$ arising in applications (see \cite{Bar74, Glad05}). Similarly to Lemma~\ref{lem:30} and Theorem~\ref{thm:30}, we obtain the following result.

\begin{thm} \label{thm:40}
For a sequence $\{ \la_{l,k}, \be_{l,k} \}_{l \ge 1, \, k = 1, 2, 3}$ to be the spectral data of $\tau = \{ \tau_0, 0, \tau_2 \} \in \mathbf{W}_{simp}^+$ (i.e. $\tau_1 \equiv 0$), it is necessary and sufficient to fulfill the conditions of Theorem~\ref{thm:char4} with $\sigma = \sigma_0 = \sigma_1 = 0$ and the following additional condition. For each $l \ge 1$, there exists an index $p(l)$ such that $\la_{p(l),1} = \overline{\la_{l,1}}$ and $\be_{p(l), 1} = \overline{\be_{l,1}}$. In particular, for all sufficiently large $l$, the values $\la_{l,1}$ and $\be_{l,1}$ are real.
\end{thm}

Note that, under the conditions of Theorem~\ref{thm:40}, there holds $\{ \la_{l,1}, \be_{l,1} \}_{l \ge 1} = \{ \la_{l,3}, \be_{l,3} \}_{l \ge 1}$ in view of \eqref{sa}. 

Proceed to Inverse Problem~\ref{ip:sp} by the eigenvalue set $\{ \la_{l,k}, \mu_{l,k} \}_J$.
For $n = 4$, $\{ \mu_{l,k}\}_{l \ge 1}$ are the eigenvalues of the problems $\mathcal M_k$ ($k = 1, 2, 3$) for equation \eqref{eqv4} with the boundary conditions:
\begin{align} \label{bcmu1}
\mathcal M_1 \colon & \quad y'(0) = 0, \quad y(1) = y'(1) = y''(1) = 0, \\ \nonumber
\mathcal M_2 \colon & \quad y(0) = y''(0) = 0, \quad y(1) = y'(1) = 0, \\ \nonumber
\mathcal M_3 \colon & \quad y(0) = y'(0) = y'''(0) = 0, \quad y(1) = 0.
\end{align}

The following theorem is proved in Appendix~\ref{app:calc4}.

\begin{thm} \label{thm:mu4}
The eigenvalues $\{\mu_{l,k}\}_J$ for $\tau = \{ \tau_0, \tau_1, \tau_2\} \in \mathbf{W}$ satisfy the asymptotics
\begin{multline} \label{asymptmu41}
\mu_{l,2\pm1} = -\Biggl( \bigl(\sqrt 2 \pi l)^4 - \theta \bigl( \sqrt 2 \pi l\bigr)^2 + \biggl( \frac{\theta_1 - \theta_0}{\sqrt 2} \mp 2 \sqrt 2 \sigma \biggr)\bigl( \sqrt 2 \pi l \bigr) \\ \pm \sigma_1 +
\frac{\theta_1' - 3 \theta_0'}{4} + \phi + \frac{\theta^2}{8} + \kappa_{l,2\pm 1} \Biggr),
\end{multline}
\vspace*{-0.8cm}
\begin{multline} \label{asymptmu42}
\mu_{l,2} = \biggl( \pi l + \frac{\pi}{4}\biggr)^4 - \theta \biggl( \pi l + \frac{\pi}{4}\biggr)^2 + \theta_1 \biggl( \pi l + \frac{\pi}{4}\biggr) - \frac{3 \theta_1' - \theta_0'}{4} - \phi + \frac{\theta^2}{8} + \kappa_{l,2},   
\end{multline}
where $\{ \kappa_{l,k}\} \in l_2$ and the constants $\theta$, $\theta_0$, $\theta_1$, $\theta_0'$, $\theta_1'$, $\sigma$, $\sigma_1$, and $\phi$ are defined in \eqref{const4}.
\end{thm}

Theorems~\ref{thm:scsp}, \ref{thm:char4}, and~\ref{thm:mu4} imply the following sufficient conditions for the existence of solution of Inverse Problem~\ref{ip:sp} in the case $n = 4$.

\begin{thm} \label{thm:sc4sp}
Let a sequence $\{ \la_{l,k}, \mu_{l,k}\}_{l \ge 1, k = 1, 2, 3}$ satisfy the conditions (A-1), (A-2), \eqref{sepsp}, \eqref{sasp}, \eqref{inter} and the asymptotics \eqref{asymptla41}, \eqref{asymptla42}, \eqref{asymptmu41}, \eqref{asymptmu42} with real constants $\theta$, $\theta_0$, $\theta_1$, $\theta_0'$, $\theta_1'$, $\phi$, purely imaginary constants $\sigma$, $\sigma_0$, $\sigma_1$, and remainder terms such that $\{ \varkappa_{l,k}\} \in l_2$, $\{ \kappa_{l,k} \ln(l + 1) \} \in l_2$. Then $\{ \la_{l,k}, \mu_{l,k}\}_{l \ge 1, \, k = 1, 2, 3}$ is the eigenvalue set of a unique vector $\tau = \{ \tau_0, \tau_1, \tau_2 \} \in \mathbf{W}_{simp}^+$.
\end{thm}

Theorems~\ref{thm:unisp} and~\ref{thm:sc4sp} together imply the uniform stability of Inverse Problem~\ref{ip:sp}.
The corresponding theorem is formulated by using the eigenvalues $\{ \la_{l,k}, \mu_{l,k} \}$ only for $k = 1, 2$, since $\la_{l,3}$ and $\mu_{l,3}$ are uniquely specified by \eqref{sasp}.

\begin{thm} \label{thm:uni4sp}
Let $\Omega > 0$ and $\de > 0$ be fixed, and let each of two sequences $\{ \la_{l,k,(1)}, \mu_{l,k,(1)}\}_{l \ge 1\, k = 1, 2}$ and $\{ \la_{l,k,(2)}, \mu_{l,k,(2)}\}_{l \ge 1\, k = 1, 2}$ satisfy 
\begin{gather*}
|\la_{l,k_1} - \la_{s,k_2}| \ge \de, \quad (l,k_1) \ne (s,k_1), \quad l,s \ge 1, \quad k_1, k_2 \in \{ 1, 2 \}, \\
|\la_{l,k}| \ge \de, \quad |\la_{l,k} - \mu_{s,k}| \ge \de, \quad \la_{l,2}, \, \mu_{l,2} \in \mathbb R,
\end{gather*}
the conditions \eqref{inter}
and the asymptotics \eqref{asymptla41}, \eqref{asymptla42}, \eqref{asymptmu41}, \eqref{asymptmu42}
with the same real constants $\theta$, $\theta_0$, $\theta_1$, $\theta_0'$, $\theta_1'$, $\phi$, imaginary constants $\sigma$, $\sigma_0$, $\sigma_1$, and with remainders satisfying 
$$
\| \{ \varkappa_{l,k} \}_{l \ge 1} \|_{l_2} \le \Omega, \quad
\| \{ \kappa_{l,k} \ln(l + 1) \}_{l \ge 1} \|_{l_2} \le \Omega, \quad k = 1, 2.
$$
Then the estimate \eqref{unisp} holds for $\nu = 0, 1, 2$.
\end{thm}

Analogously to Theorem~\ref{thm:40}, we obtain sufficient conditions on the eigenvalue set for the special case $\tau_1 \equiv 0$.

\begin{thm} \label{thm:40sp}
Let a sequence $\{ \la_{l,k}, \mu_{l,k}\}_{l \ge 1, \, k = 1, 2, 3}$ satisfy the conditions of Theorem~\ref{thm:sc4sp} with $\sigma = \sigma_0 = \sigma_1 = 0$ and the sets $\{ \la_{l,1}\}_{l \ge 1}$ and $\{ \mu_{l,1}\}_{l \ge 1}$ are symmetric with respect to the real axis. Then $\{ \la_{l,k}, \mu_{l,k}\}_{l \ge 1, \, k = 1, 2, 3}$ is the eigenvalue set of a unique vector $\tau = \{ \tau_0, 0, \tau_2\} \in \mathbf{W}_{simp}$ (i.e. $\tau_1 \equiv 0$).
\end{thm}

The conditions of Theorem~\ref{thm:40sp}, in particular, imply that $\{ \la_{l,1}, \mu_{l,1}\}_{l \ge 1} = \{ \la_{l,3}, \mu_{l,3}\}_{l \ge 1}$ and these values are real for sufficiently large $l$.

Inverse Problems~\ref{ip:main} and~\ref{ip:sp} in the case $n = 4$, $\tau_1 \equiv 0$ are uniformly stable according to Theorems~\ref{thm:uni4} and~\ref{thm:uni4sp}.

\begin{remark}
The examples of $n = 3, 4$ show that, for the eigenvalue set $\{ \la_{l,k}, \mu_{l,k} \}_J$, there is a gap ``$\ln(l + 1)$'' in the asymptotics between the necessary conditions and our sufficient conditions (Theorems~\ref{thm:char3sp} and~\ref{thm:sc4sp}). This logarithm is essential in the proof of the estimate \eqref{aD2} (see Remark~\ref{rem:ln}). In the case $n = 2$, Freiling and Yurko \cite[Section~1.5.3]{FY01} have obtained the analog of \eqref{aD2} without requiring the additional $\ln(l+1)$, by relying on the assertion that $\Delta_{k+1,k}(\la)$ is the characteristic function of a suitable Sturm-Liouville operator. For higher orders $n$, this approach requires much technical work: one has to repeat all the proofs from \cite{Bond24} under different boundary conditions. Therefore, in this paper, we confine ourselves by sufficient conditions containing the additional logarithm. Anyway, this paper does not aim to give a complete characterization of the eigenvalue set, since we impose the other restrictive conditions (A-1) and (A-2).
\end{remark}

\renewcommand{\thesection}{\Alph{section}}
\setcounter{section}{0}

\section{Estimates for Weyl solutions} \label{app:Weyl}

In this appendix, we recall asymptotic properties of the Weyl solutions $\Phi_k(x,\la)$ from \cite{Yur02, Bond21}. Moreover, we present the estimates for $\vv_{l,k,\eps}(x)$ and other functions of Sections~\ref{sec:maineq} and~\ref{sec:bound} from previous studies \cite{Bond22, Bond23-loc} and discuss the conditions under which these estimates are uniform. 

Put $\la = \rho^n$ and divide the $\rho$-plane into the sectors
\begin{equation} \label{defGa}
\Gamma_s := \left\{ \rho \in \mathbb C \colon \frac{\pi(s-1)}{n} < \arg \rho < \frac{\pi s}{n}\right\}, \quad s = \overline{1,2n}.
\end{equation}

In each fixed sector $\Gamma_s$, denote by $\{ \om_k \}_{k = 1}^n$ the roots of the equation $\om^n = 1$ numbered so that
\begin{equation} \label{order}
\mbox{Re} \, (\rho \om_1) < \mbox{Re} \, (\rho \om_2) < \dots < \mbox{Re} \, (\rho \om_n), \quad \rho \in \Gamma_s.
\end{equation}

For $s \in \{ 1, 2, \dots, 2n \}$, $h > 0$, and $\rho_* > 0$, define the region
$$
\Gamma_{s,h,\rho_*} :=  \left\{ \rho \in \mathbb C \colon \rho + h \exp\bigl( \tfrac{i \pi (s - 1/2)}{n}\bigr)  \in \Gamma_s, \, |\rho| > \rho_*\right\}. 
$$

In order to estimate the Weyl solutions $\Phi_k(x,\rho)$, we use the Birkhoff solutions described by the following proposition.

\begin{prop}[\hspace*{-5pt}\cite{Nai68}] \label{prop:Birk}
For every $s \in \{1,2, \dots, 2n\}$, $h > 0$, and some $\rho_* > 0$, there exists a fundamental system of solutions $\{ y_k(x,\rho) \}_{k = 1}^n$ of equation \eqref{eqp} with the following properties:
\begin{enumerate}
\item For each $k \in \{ 1, 2, \dots, n \}$, the function $y_k(x,\rho)$ is continuous for $x \in [0,1]$, $\rho \in \overline{\Gamma}_{s,h,\rho_*}$ and satisfy the equation
\begin{multline} \label{Volt}
y_k(x,\rho) = \exp(\rho \om_k x) - \int_0^x \sum_{j = 1}^k \om_j \exp(\rho \om_j (x-t)) \mathcal M_t(y_k)\,dt \\ + 
\int_x^1 \sum_{j = k+1}^n \om_j \exp(\rho\om_j(x-t)) \mathcal M_t(y_k)\, dt, \quad \mathcal M_t(y_k) := \frac{1}{n} \rho^{1-n} \sum_{\mu = 0}^{n-2} p_{\mu}(t) y_k^{(\mu)}(t,\rho).
\end{multline}
\item For each fixed $x \in [0,1]$, the functions $y_k(x,\rho)$ are analytic in $\Gamma_{s,h,\rho_*}$.
\item The estimate
\begin{equation} \label{estBirk}
\bigl| y_k^{(\nu)}(x,\rho) (\rho \om_k)^{-\nu} \exp(-\rho \om_k x) - 1 \bigr| \le C |\rho|^{-1},
\end{equation}
hold uniformly with respect to $x \in [0,1]$ and $\rho \in \overline{\Gamma}_{s,h,\rho_*}$.
\end{enumerate}
\end{prop}

Assume that $\mathbf{p} = \{ p_s \}_{s = 0}^{n-2}$ in $B_Q$ for some $Q > 0$, that is, $\| p_s \|_{W_2^s[0,1]} \le Q$ for $s = \overline{0,n-2}$. Then the value $\rho_*$ in Proposition~\ref{prop:Birk} and the constant $C$ in \eqref{estBirk} depend only on $h$ and $Q$.

For $\rho \in \overline{\Gamma}_{s,h,\rho_*}$, the Weyl solutions can be expanded as
\begin{equation} \label{expPhik}
\Phi_k(x,\la) = \sum_{j = 1}^n b_{j,k}(\rho) y_j(x,\rho),
\end{equation}
where the coefficients $b_{j,k}(\rho)$ are analytic for $\rho^n \ne \la_{l,k}$.

Denote by $\{ \la_{l,k}^0 \}_{l \ge 1}$ the eigenvalues of the boundary value problems $\mathcal L_k^0$ for the equation $y^{(n)} = \lambda y$ with the boundary conditions \eqref{bc}. For sufficiently large $l$, the eigenvalues $\{ \la_{l,k}^0 \}$ are real (see \cite[Lemma~3]{Bond24}). 
Let $\rho_*$ be sufficiently large and $l_*$ be such that $|\la_{l,k}^0| \ge \rho_*^n$ for all $l \ge l_*$.
Then, for $l \ge l_*$,
one can choose the roots $\rho_{l,k}^0 := \sqrt[n]{\la_{l,k}^0}$ lying in $\Gamma_{s,h,\rho_*}$. So, define the regions
$$
\Gamma_{s,h,\rho_*}^{k,\de} := \bigl\{ \rho \in \Gamma_{s,h,\rho_*} \colon \forall l \ge l_* \: |\rho - \rho_{l,k}^0| \ge \de \}.
$$

In view of the asymptotics \eqref{asymptla}, there holds $\rho_{l,k} - \rho_{l,k}^0 = o(1)$ as $l \to \infty$, where $\rho_{l,k} = \sqrt[n]{\la_{l,k}} \in \Gamma_{s,h,\rho_*}$. Hence, for sufficiently large $\rho_*$ and sufficiently small $h$ and $\de$, the function $\Phi_k(x, \rho^n)$ is analytic in $\Gamma_{s,h,\rho_*}^{k,\de}$.

By using Proposition~\ref{prop:Birk} and \eqref{expPhik}, the following lemma is deduced.

\begin{lem} \label{lem:estPhik}
There hold
\begin{gather} \label{estPhik}
|\Phi_k^{(j-1)}(x,\rho^n)| \le C |\rho|^{j-k} |\exp(\rho \om_k x)|, \\
\label{estdifPhik} |\Phi_k^{(j-1)}(x,\rho^n) - \tilde \Phi_k^{(j-1)}(x,\rho^n)| \le C |\rho|^{j-k-1} |\exp(\rho \om_k x)|, 
\end{gather}
for $k,j = \overline{1,n}$, $x \in [0,1]$, $\rho \in \overline{\Gamma}_{s,h,\rho_*}^{k,\de}$, $s = \overline{1,2n}$, any $h > 0$ and a sufficiently small $\de > 0$, and some $\rho^* > 0$. 
For $\mathbf{p}$ and $\tilde{\mathbf{p}}$ in $B_Q$, any $h > 0$, and any sufficiently small $\de > 0$, the constants $\rho_*$ and $C$ can be chosen depending only on $Q$, $h$, and $\de$.
\end{lem}

\begin{proof}[Sketch of the proof]
The estimates \eqref{estPhik} and \eqref{estdifPhik} have been obtained in the previous studies \cite{Yur02, Bond21, Bond24}
for the non-extended closed sectors $\overline{\Gamma}_s$ with appropriate ``holes'', see the asymptotics (2.1.19) in \cite{Yur02}, and a more detailed derivation in \cite[Section~5]{Bond21}, and \cite[Proposition~1]{Bond24} for the derivatives of $\Phi_k$. The asymptotics of the solution $\Phi_k(x,\la)$ and its derivatives are deduced by using the representation \eqref{expPhik} and the asymptotics of the Birkhoff solutions from Proposition~\ref{prop:Birk}. In particular, the coefficients $b_{j,k}(\rho)$ have the form $\dfrac{d_{j,k}(\rho)}{d_k(\rho)}$, where $d_{j,k}(\rho)$ and $d_k(\rho)$ are some determinants composed of the Birkhoff solutions $\{ y_r(x,\rho) \}_{r=1}^n$ and their derivatives at the points $x = 0$ and $x = 1$.
Since the estimates \eqref{estBirk} are valid in the extended sectors $\overline{\Gamma}_{s,h,\rho_*}$, then the estimates \eqref{estPhik} and \eqref{estdifPhik} are easily transferred to the desired regions $\overline{\Gamma}_{s,h,\rho_*}^{k,\de}$.
\end{proof}

Using Lemma~\ref{lem:estPhik}, we obtain the estimates for the functions $\vv_{l,k,\eps}(x)$ and $\dot \vv_{l,k,\eps}(x)$ defined by \eqref{defvv} and \eqref{defdvv}, respectively, summarized in the following lemma.

\begin{lem} \label{lem:estvv}
Suppose that $\Omega, \de > 0$, $\mathbf{p}$ and $\tilde{\mathbf{p}}$ belong to $B_Q$ and the corresponding eigenvalues satisfy \eqref{bbound} and $|\la_{l,k} - \la_{s,k+1}| \ge \de$ for all $l,s \ge 1$, $k = \overline{1,n-2}$. Then
\begin{gather} \label{estvv1}
|\vv_{l,k,\eps}^{(j)}(x)| \le C l^j w_{l,k}(x), \\ \label{estvv2}
|\dot \vv_{l,k,\eps}^{(j)}(x)| \le C l^j w_{l,k}(x), \\ \label{estvv3}
|\vv_{l,k,0}^{(j)}(x) - \vv_{l,k,1}^{(j)}(x)| \le C l^j \xi_l w_{l,k}(x).
\end{gather}
If additionally $|\tilde \la_{l,k} - \tilde \la_{s,k+1}| \ge \de$ for all $l,s \ge 1$, $k = \overline{1,n-2}$, then
\begin{gather} \label{estvv4}
|\vv_{l,k,\eps}^{(j)}(x) - \tilde \vv_{l,k,\eps}^{(j)}(x)| \le C l^{j-1} w_{l,k}(x), \\ \label{estvv5}
|\dot \vv_{l,k,\eps}^{(j)}(x) - \dot{\tilde \vv}_{l,k,\eps}^{(j)}(x)| \le C l^j w_{l,k}(x), \\ \label{estvv6}
|\vv_{l,k,0}^{(j)}(x) - \vv_{l,k,1}^{(j)}(x) - \tilde \vv_{l,k,0}^{(j)}(x) + \tilde \vv_{l,k,1}^{(j)}(x)| \le C l^{j-1} \xi_l w_{l,k}(x).
\end{gather}
In the above estimates $(l,k,\eps) \in V$, $x \in [0,1]$, $j = \overline{0,n-1}$, $w_{l,k}(x)$ and $\xi_l$ are defined by \eqref{defw} and \eqref{defxi}, respectively, and the constant $C$ depends only on $Q$ and $\de$.
\end{lem}

\begin{proof}
Fix any $k \in \{ 1, 2, \dots, n-2\}$, $s \in \{1, 2, \dots, 2n\}$, and $h > 0$. Then, under the assumptions of this lemma, the values $\sqrt[n]{\la_{l,k,\eps}}$ for sufficiently large $l$ lie in $\overline{\Gamma}_{s,h,\rho_*}^{k+1,\de}$, where the root branch is chosen so that $\sqrt[n]{\la_{l,k,\eps}} \in \Gamma_{s,h,\rho_*}$. Indeed, according to the asymptotics \eqref{asymptla}, we have $\mbox{Re} \, \la_{l,k} \to +\infty$ if $(n-k)$ is even and $\mbox{Re} \, \la_{l,k} \to -\infty$ if $(n-k)$ is odd. Therefore, the eigenvalues $\la_{l,k}$ are asymptotically separated from $\la_{s,k+1}^0$. Hence, one immediately obtains the estimates \eqref{estvv1} and \eqref{estvv4} for $l$ not less than some $l_*$ from \eqref{estPhik} and \eqref{estdifPhik}, respectively.

For $l < l_*$, the eigenvalues $\la_{l,k,\eps}$ lie in a circle $|\la| \le r$, whose radius depends only on $Q$. The functions $\Phi_k^{(j-1)}(x,\la)$ are continuous by $x \in [0,1]$ and $\la \in \mathbb C \setminus \{ \la_{l,k} \}$, so
\begin{equation} \label{estsmall}
|\Phi_k^{(j)}(x,\la)| \le C, \quad |\la| \le r, \quad |\la - \la_{l,k}| \ge \de,
\end{equation}
for $j = \overline{0,n-1}$ and $x \in [0,1]$, where the constant $C$ depends only on $Q$ and $\de$. The estimate \eqref{estsmall} imply \eqref{estvv1} and \eqref{estvv4} for $l < l_*$.

The estimates \eqref{estvv2}, \eqref{estvv3}, \eqref{estvv5}, and \eqref{estvv6} are obtained by applying Schwarz's Lemma (see Lemma 1.3.1 and its application in the proof of Lemma 1.3.2 in \cite{Yur02}) to \eqref{estPhik}, \eqref{estdifPhik}, and \eqref{estsmall}.
\end{proof}

Clearly, Lemma~\ref{lem:estPhik} is applicable to the Weyl solutions $\Phi_k^{\star}(x,\la)$ of equation \eqref{eqstar}. Under the assumptions of Lemma~\ref{lem:estvv}, we get the following estimates (see \cite[Corollary 2]{Bond23-loc}):
\begin{equation} \label{estPhi*}
|\Phi_{n-k+1}^{\star(j)}(x, \la_{l,k,\eps})| \le C l^{j-n} w_{l,k}^{-1}(x), \quad
|\Phi_{n-k+1}^{\star(j)}(x, \la_{l,k,0}) - \Phi_{n-k+1}^{\star(j)}(x, \la_{l,k,1})| \le C l^{\nu - n} w_{l,k}^{-1}(x) \xi_l,
\end{equation}
where $l \ge 1$, $k = \overline{1,n-1}$, $x \in [0,1]$, and $C = C(Q,\de)$. Using \eqref{estPhi*} together with \eqref{asymptbe} and \eqref{defxi}, we obtain Lemma~\ref{lem:eta}.

Applying \eqref{asymptbe}, \eqref{estPhi*}, and the estimates of Lemma~\ref{lem:estPhik} to the definitions of the functions $\psi_v(x)$, $\tilde \psi_v(x)$, and $\tilde R_{v_0, v}(x)$ from Section~\ref{sec:maineq}, we arrive at \eqref{estpsiR}. Suppose that the model vector $\tilde \tau \in W_{simp}$ is fixed and the data $\{ \la_{l,k}, \be_{l,k} \}_J$ satisfy the assumptions $\Xi \le \Omega$ and \eqref{bbound}. Then the constant $C$ in the estimates \eqref{estpsiR} for $\tilde \psi_v(x)$ and $\tilde R_{v_0,v}(x)$ depends only on $\Omega$ and $\de$. Under these conditions, the relation \eqref{estop} implies \eqref{uniR}.

\section{Spectral data asymptotics} \label{app:calc}

In this appendix, we derive the asymptotic relations for the spectral characteristics $\la_{l,k}$, $\mu_{l,k}$, and $\be_{l,k}$ in the cases $n = 3$ and $n = 4$
by using the standard method \cite{Nai68}.

\subsection{Basic strategy}

Let $n \ge 2$ and $k \in \{1, 2, \dots, n-1 \}$ be fixed. Recall that the eigenvalues $\{ \la_{l,k} \}_{l \ge 1}$ of the problem $\mathcal L_k$ are the zeros of the determinant $\Delta(\la) := \Delta_{k,k}(\la)$ given by \eqref{defDelta}.

Consider the sector $\Gamma_1$ given by \eqref{defGa} and the Birkhoff solutions $y_k(x,\rho)$ ($k = \overline{1,n}$) defined in the extended sector $\rho \in \Gamma_{1,h,\rho_*}$ according to Proposition~\ref{prop:Birk}. 
Suppose that $\la = \rho^n$, $\rho \in \Gamma_{1,h,\rho_*}$.
Expand the solutions $\mathcal C_k(x,\la)$ ($k = \overline{1,n}$) over the Birkhoff solutions:
\begin{equation} \label{expandC}
[\mathcal C_k^{(j-1)}(x,\la)]_{k,j = 1}^n = [y_k^{(j-1)}(x,\rho)]_{k,j=1}^n A(\rho), 
\end{equation}
where $A(\rho)$ is an $(n \times n)$ matrix function. Then, we get the relation
\begin{equation} \label{DelD}
\Delta(\la) = D(\rho) \det A(\rho),
\end{equation}
where $D(\rho)$ is some determinant composed of the functions $y_k^{(\nu)}(x,\rho)$ at $x \in \{ 0, 1\}$ and $\det A(\rho) \ne 0$ for sufficiently large $|\rho|$, $\rho \in \Gamma_{1,h,\rho_*}$. Consequently, for large $l$, there holds $\la_{l,k} = \rho_{l,k}^n$, where $\rho_{l,k}$ are the zeros of $D(\rho)$. Thus, it remains to find the asymptotics of $D(\rho)$ and its zeros.

Analogously, the eigenvalues $\{ \mu_{l,k} \}_{l \ge 1}$ of the problem $\mathcal M_k$ \eqref{eqv}, \eqref{bcmu} coincide with the zeros of $\Delta^+(\la) := \Delta_{k+1,k}(\la)$. Using \eqref{expandC}, we get the representation
\begin{equation} \label{DelDp}
\Delta^+(\la) = -D^+(\rho) \det A(\rho),
\end{equation}
so it remains to find the asymptotics of the zeros of $D^+(\rho)$.

Proceed to the weight numbers.
For large $l$, the eigenvalues $\la_{l,k}$ are simple zeros of $\Delta(\la)$, so the relation \eqref{beDe} holds.
Taking \eqref{DelD} and \eqref{DelDp} into account, we obtain
\begin{equation} \label{calcbe}
\be_{l,k} = n \rho_{l,k}^{n-1} \frac{D^+(\rho_{l,k})}{\frac{d}{d\rho} D(\rho_{l,k})}.
\end{equation}

In this appendix, we use the notation
$\Upsilon(\rho)$ for various functions satisfying $\Upsilon(\rho) \to 0$ as $|\rho| \to \infty$ in $\Gamma_{1,h,\rho_*}$ and $\{ \Upsilon(\rho_l) \} \in l_2$ for any non-condensing sequence $\{ \rho_l \} \subset \Gamma_{1,h,\rho_*}$. A sequence $\{ \rho_l \}_{l = 1}^{\infty}$ is called non-condensing if
$$
\sup_{r > 0}(N(r+1) - N(r)) < \infty, \quad N(r) := \# \{ l \in \mathbb N \colon |\rho_l| \le r \}.
$$

The notation $\Upsilon(x,\rho)$ denotes any function that possesses the same properties as $\Upsilon(\rho)$ for each fixed $x \in [0,1]$.

In order to derive asymptotic formulas for $D(\rho)$ and $D^+(\rho)$, we use the well-known asymptotics for the Birkhoff solutions and their derivatives (see, e.g., \cite[\S4]{Nai68}):
\begin{equation} \label{asympty}
\def\arraystretch{2.2}
\left.
\begin{array}{l}
y_k(x,\rho) = \exp(\rho \om_k x) \left( 1 + \displaystyle\sum\limits_{s = 1}^{n-1} \dfrac{q_s(x)}{(\rho \om_k)^s} + \eps(x,\rho)\right), \\
y_k^{(\nu)}(x,\rho) = (\rho \om_k)^{\nu} \exp(\rho \om_k x) \left( 1 + \displaystyle\sum\limits_{s = 1}^{n-1} \dfrac{q_{s \nu}(x)}{(\rho \om_k)^s} + \eps(x,\rho) \right), \quad \nu = \overline{1,n-1},
\end{array} \quad \right\}
\end{equation}
for $k = \overline{1,n}$, $|\rho| \to \infty$, $\rho \in \Gamma_{1,h,\rho_*}$. Here the notation $\eps(x,\rho)$ means various functions of form $\dfrac{\Upsilon(x,\rho)}{\rho^{n-1}}$. This form of the remainder terms in \eqref{asympty} follows from the results of \cite{KS24}. We use the formulas for the functions $q_s(x)$ and $q_{s\nu}(x)$ presented in \cite[Theorem~1.2]{Yur22}:
\begin{equation} \label{defq}
\def\arraystretch{1.5}
\left.
\begin{array}{c}
q_1'(x) = -\dfrac{1}{n} p_{n-2}(x), \quad
q_2'(x) = -\dfrac{1}{n} \bigl( C_n^2 q_1''(x) + p_{n-2}(x) q_1(x) + p_{n-3}(x)\bigr), \\
q_3'(x) = -\dfrac{1}{n} \Bigl( C_n^2 q_2''(x) + C_n^3 q_1'''(x) + p_{n-2}(x)\bigl(q_2(x) + C_{n-2}^1 q_1'(x)\bigr) + p_{n-3}(x) q_1(x) + p_{n-4}(x)\Bigr), \\
q_{s\nu}(x) = \displaystyle\sum\limits_{r = 0}^{\nu} C_{\nu}^r q_{s-r}^{(r)}(x), \quad 
q_0(x) = 1, \quad q_s(x) = 0 \: \text{for} \: s < 0.
\end{array} \right\}
\end{equation}
 
Integration constants for the differential equations in \eqref{defq} can be chosen arbitrarily.

For finding the coefficients in the asymptotic expansions of the characteristic functions and their zeros, we use symbolic computations in Python \cite{py}.

By virtue of \cite[Lemma~2]{Bond22}, the Weyl functions associated with equations \eqref{eqv} and \eqref{eqstar} satisfy the relation $M_k(\la) = M_{n-k}^{\star}(\la)$ for $k = \overline{1,n-1}$. Consequently, in view of \eqref{defbe} and \eqref{relmjk}, there holds
\begin{equation} \label{relstar}
\la_{l,k} = \la_{l,n-k}^{\star}, \quad \mu_{l,k} = \mu_{l,n-k}^{\star}, \quad \be_{l,k} = \be_{l,n-k}^{\star}, \quad k = \overline{1,n-1},
\end{equation}
which allows us to reduce the amount of computations.

\subsection{Case $n = 3$} \label{app:calc3}

This subsection contains the proof of Theorem~\ref{thm:asymptsd3}. Here, we use the notations:
\begin{itemize}
\item $\om_1 = \exp(2 \pi i/3)$, $\om_2 = \exp(-2 \pi i/3)$, and $\om_3 = 1$ are the roots $\sqrt[3]{1}$ satisfying \eqref{order} for $\rho \in \Gamma_1$.
\item $\eps(\rho) = \dfrac{\Upsilon(\rho)}{\rho^2}$.
\item $\{ \eps_l \}$ denotes various sequences such that $\{ l^2 \eps_l \} \in l_2$.
\item $t := \frac{1}{3} \int_0^1 \tau_1(x) \, dx$, $t_0 := \frac{1}{3} \tau_1(0)$, $t_1 := \frac{1}{3} \tau_1(1)$, $s := \frac{1}{3} \int_0^1 \tau_0(x)\,dx$.
\end{itemize}

Using \eqref{defq} and the relations $p_1 = 2 \tau_1$, $p_0 = \tau_1' + \tau_0$, we find
\begin{gather*}
q_1(x) = -\frac{2}{3} \int_0^x \tau_1(\xi) \, d\xi, \quad q_2(x) = \frac{2}{9} \left( \int_0^x \tau_1(\xi) \, d\xi\right)^2 + \frac{1}{3} \tau_1(x) - \frac{1}{3} \int_0^x \tau_0(\xi) d\xi, \\ 
q_{11}(x) = q_1(x), \quad q_{21}(x) = q_2(x) - \frac{2}{3} \tau_1(x).
\end{gather*}

According to our notations, we have
\begin{equation} \label{valq}
\def\arraystretch{1.3}
\left.\begin{array}{c}
q_1(0) = q_{11}(0) = 0, \quad q_2(0) = t_0, \quad q_{21}(0) = -t_0,  \\
q_1(1) = q_{11}(1) = -2t, \quad q_2(1) = 2 t^2 + t_1 - s, \quad q_{21}(1) = 2 t^2 - t_1 - s.
\end{array}\quad\right\}
\end{equation}

The eigenvalues $\{ \la_{l,1} \}_{l \ge 1}$ of the problem $\mathcal L_1$ \eqref{eqv3}--\eqref{bc31} coincide with the zeros of the characteristic determinant
$$
\Delta(\la) := \Delta_{1,1}(\la) = \begin{vmatrix}
\mathcal C_1(0,\la) & \mathcal C_2(0,\la) & \mathcal C_3(0,\la) \\
\mathcal C_1'(1,\la) & \mathcal C_2'(1,\la) & \mathcal C_3'(1,\la) \\
\mathcal C_1(1,\la) & \mathcal C_2(1,\la) & \mathcal C_3(1,\la)
\end{vmatrix}.
$$

The relation \eqref{DelD} holds with
\begin{equation} \label{defD1}
D(\rho) := \begin{vmatrix}
                y_1(0,\rho) & y_2(0,\rho) & y_3(0,\rho) \\
                y_1'(1,\rho) & y_2'(1,\rho) & y_3'(1,\rho) \\
                y_1(1,\rho) & y_2(1,\rho) & y_3(1,\rho)
           \end{vmatrix}.
\end{equation}

For brevity, denote $\om := \om_1$.
Substituting \eqref{asympty} for $n = 3$ together with \eqref{valq} into \eqref{defD1}, we get
\begin{align*}
& D(\rho) = \rho \exp(\rho (\om_2 + \om_3)) d(\rho), \\
& d(\rho) = r_1(\rho) - r_2(\rho) \exp(\rho(\om_1-\om_2)) + \eps(\rho), \\
& r_1(\rho) = \om^2 - 1 + \frac{2t(\om - \om^2)}{\rho} + \frac{(1-\om)(s + 2t^2 + t_0 + t_1)}{\rho^2}, \\
& r_2(\rho) = \om-1 + \frac{2 t(\om^2 - \om)}{\rho} + \frac{(1-\om^2)(s + 2t^2 + t_0 + t_1)}{\rho^2}.
\end{align*}

The zeros of $d(\rho)$ for large $|\rho|$ satisfy the relation
\begin{equation} \label{rho21}
\rho = \frac{1}{\om_2 - \om_1} \ln \frac{r_2(\rho)}{r_1(\rho)} + \eps(\rho).
\end{equation}
Consequently, we get
$$
\rho_{l,1} = \frac{1}{2 i \sin \frac{\pi}{3}} \left( 2 \pi i l + \frac{\pi i}{3} + \frac{2t(\om^2 - \om)}{\rho_{l,1}} + \frac{(s + t_0 + t_1)(\om - \om^2)}{\rho_{l,1}^2} + \eps_l \right),
$$
which implies
\begin{equation*}
\rho_{l,1} = \frac{\pi}{\sin\frac{\pi}{3}} \left( l + \frac{1}{6} - \frac{\theta}{2 \pi^2 \bigl(l + \frac{1}{6}\bigr)} + \frac{\sqrt 3 (\sigma + \theta_0 + \theta_1)}{8 \pi^3 l^2} + \eps_l \right),
\end{equation*}
according to the notations \eqref{const3}. Finding $\rho_{l,1}^3$, we arrive at \eqref{asymptla31}. 

\smallskip

Proceed to obtaining the asymptotics for $\mu_{l,1}$. We have
$$
\Delta^+(\la) := \Delta_{2,1}(\la) = 
-\begin{vmatrix}
\mathcal C_1'(0,\la) & \mathcal C_2'(0,\la) & \mathcal C_3'(0,\la) \\
\mathcal C_1'(1,\la) & \mathcal C_2'(1,\la) & \mathcal C_3'(1,\la) \\
\mathcal C_1(1,\la) & \mathcal C_2(1,\la) & \mathcal C_3(1,\la)
\end{vmatrix}, \quad
D^+(\rho) = 
\begin{vmatrix}
y_1'(0,\rho) & y_2'(0,\rho) & y_3'(0,\rho) \\
y_1'(1,\rho) & y_2'(1,\rho) & y_3'(1,\rho) \\
y_1(1,\rho) & y_2(1,\rho) & y_3(1,\rho)
\end{vmatrix}.
$$

Substituting \eqref{asympty} and \eqref{valq} into the formula for $D^+(\rho)$, we derive
\begin{align*}
& D^+(\rho) = \rho^2 \exp(\rho (\om_2 + \om_3)) d^+(\rho), \\
& d^+(\rho) = r_1^+(\rho) - r_2^+(\rho) \exp(\rho(\om_1 - \om_2)) + \eps(\rho), \\
& r_1^+(\rho) = 1 - \om + \frac{2t(\om^2 - 1)}{\rho} + \frac{(\om-\om^2)(s + 2t^2 - t_0 + t_1)}{\rho^2}, \\
& r_2^+(\rho) = 1 - \om^2 + \frac{2t(\om-1)}{\rho} + \frac{(\om^2 - \om)(s + 2t^2 - t_0 + t_1)}{\rho^2}.
\end{align*}

The zeros of $d^+(\rho)$ for large $|\rho|$ satisfy the relation
\begin{equation} \label{rho21p}
\rho = \frac{1}{\om_2 - \om_1} \ln \frac{r_2^+(\rho)}{r_1^+(\rho)} + \eps(\rho).
\end{equation}
Consequently, they have the form
\begin{align*}
\varrho_{l,1} & = \frac{1}{2 i \sin \frac{\pi}{3}} \left( 2 \pi i l - \frac{\pi i}{3} + \frac{2t(\om^2 - \om)}{\varrho_{l,1}} + \frac{(s - t_0 + t_1)(\om - \om^2)}{\varrho_{l,1}^2} + \eps_l \right) \\
& = \frac{\pi}{\sin\frac{\pi}{3}} \left( l - \frac{1}{6} - \frac{\theta}{2 \pi^2 \bigl(l - \frac{1}{6}\bigr)} + \frac{\sqrt 3 (\sigma - \theta_0 + \theta_1)}{8 \pi^3 l^2} + \eps_l \right)
\end{align*}
according to the notations \eqref{const3}. Computing $\mu_{l,1} = \varrho_{l,1}^3$ implies \eqref{asymptmu31}.

Note that the numbering in \eqref{asymptla31} and \eqref{asymptmu31} starts from $l = 1$, since this does not depend on $\tau$ and can be directly verified for $\tau = \{ 0, 0 \}$. Moreover, this fact for $\la_{l,1}$ is known from previous studies (see, e.g., \cite{Bond23-res}).

\smallskip

Proceed to obtaining the asymptotics for the weight numbers $\be_{l,1}$.
Taking \eqref{calcbe} and the equality $D(\rho_{l,k}) = 0$ into account, we derive
\begin{equation} \label{calcbe1}
\be_{l,k} = n \la_{l,k} \frac{d^+(\rho_{l,k})}{\frac{d}{d\rho}d(\rho_{l,k})}.
\end{equation}

Symbolic computations show that
\begin{align} \label{smcalcbe}
\frac{d^+(\rho)}{\frac{d}{d\rho} d(\rho)} & = \frac{r_1^+(\rho) r_2(\rho) - r_2^+(\rho) r_1(\rho)}{\frac{d}{d\rho}r_1(\rho) r_2(\rho) - \frac{d}{d\rho} r_2(\rho) r_1(\rho) + (\om_2 - \om_1) r_1(\rho) r_2(\rho)} \\ \nonumber & = -\left( 1 + \frac{2(t+t_0)}{\rho^2} + \eps(\rho)\right).
\end{align}

This together with \eqref{const3}, \eqref{asymptla31}, and \eqref{calcbe1} imply the asymptotic formula \eqref{asymptbe312} for $\be_{l,1}$.

The asymptotics for $\la_{l,2}$, $\mu_{l,2}$, and $\be_{l,2}$ follow from the relations \eqref{relstar}: $\la_{l,2} = \la_{l,1}^{\star}$, $\mu_{l,2} = \mu_{l,1}^{\star}$, and $\be_{l,2} = \be_{l,1}^{\star}$. The spectral characteristics $\la_{l,1}^{\star}$, $\mu_{l,1}^{\star}$, and $\be_{l,1}^{\star}$ of equation \eqref{eqstar} have the asymptotics similar to $-\la_{l,1}$, $-\mu_{l,1}$, and $-\be_{l,1}$, respectively, for the vector $\{ -\tau_0, \tau_1 \}$. Thus, one immediately gets \eqref{asymptla32}, \eqref{asymptmu32}, and \eqref{asymptbe312} for $k = 2$ from \eqref{asymptla31}, \eqref{asymptmu31}, and \eqref{asymptbe312} for $k  = 1$, respectively. Also, we have obtained the asymptotics for $\la_{l,2}$, $\mu_{l,2}$, and $\be_{l,2}$ in the straightforward way by symbolic computations in \cite{py} for checking.

\subsection{Case $n = 4$} \label{app:calc4}

In this subsection, we derive the asymptotic formulas for the data $\la_{l,k}$, $\mu_{l,k}$, and $\be_{l,k}$ associated with any vector $\tau = \{ \tau_0, \tau_1, \tau_2 \} \in \mathbf{W}$. 
Here, we use the notations:
\begin{itemize}
\item $\om_1 = -1$, $\om_2 = i$, $\om_3 = -i$, and $\om_4 = 1$ are the roots $\sqrt[4]{1}$ satisfying \eqref{order} for $\rho \in \Gamma_1$.
\item $\eps(\rho) = \dfrac{\Upsilon(\rho)}{\rho^3}$.
\item $\{ \eps_l \}$ denotes various sequences such that $\{ l^3 \eps_l \} \in l_2$.
\item $t := \dfrac{1}{4} \displaystyle\int_0^1 \tau_2(x) \, dx$, $t_0 := \dfrac{1}{8} \tau_2(0)$, $t_1 := \dfrac{1}{8} \tau_2(1)$, $t_0' := \dfrac{1}{16}\tau_2'(0)$, $t_1' := \dfrac{1}{16} \tau_2'(1)$, \\ $s := \dfrac{1}{2} \displaystyle\int_0^1 \tau_1(x) \, dx$, $s_0 := \dfrac{1}{2} \tau_1(0)$, $s_1 := \dfrac{1}{2} \tau_1(1)$, $u := \dfrac{1}{32} \displaystyle\int_0^1 \tau_2^2(x) \, dx - \dfrac{1}{4} \displaystyle\int_0^1 \tau_0(x) \, dx$.
\item $\tau_{\nu}^{(-1)}(x) := \displaystyle\int_0^x \tau_{\nu}(\xi) \, d\xi$, $\nu = 0, 1, 2$.
\end{itemize}

Integrating the equations in \eqref{defq} and taking into account the relations
$p_2 = \tau_2$, $p_1 = \tau_2' + 2 \tau_1$, $p_0 = \tau_1' + \tau_0$, we obtain
\begin{equation*}
q_1(x) = -\frac{1}{4} \tau_2^{(-1)}(x), \quad
q_2(x) = \frac{1}{8} \tau_2(x) + \frac{1}{32} \bigl( \tau_2^{(-1)}(x) \bigr)^2 - \frac{1}{2} \tau_1^{(-1)}(x), 
\end{equation*}
\vspace*{-0.5cm}
\begin{multline*}
q_3(x) = \frac{1}{16} \tau_2'(x) - \frac{1}{32} \tau_2(x) \tau_2^{(-1)}(x) + \frac{1}{2} \tau_1(x) + \frac{1}{32} \int_0^x \tau_2^2(\xi) \,d\xi \\ + \frac{1}{8} \tau_1^{(-1)}(x) \tau_2^{(-1)}(x) - \frac{1}{3 \cdot 2^7} \bigl( \tau_2^{(-1)}(x)\bigr)^3 - \frac{1}{4} \tau_0^{(-1)}(x), 
 \end{multline*}
\vspace*{-0.5cm}
\begin{gather*}
q_{1\nu}(x) = q_1(x), \quad q_{2\nu}(x) = q_2(x) + \nu q_1'(x), \quad \nu = 1, 2, \\
q_{31}(x) = q_3(x) + q_2'(x), \quad q_{32}(x) = q_3(x) + 2 q_2'(x) + q_1''(x). 
 \end{gather*}
Consequently, we have
\begin{equation} \label{valq4}
\def\arraystretch{1.3}
\left.
\begin{array}{c}
q_1(0) = q_{1\nu}(0) = 0, \quad q_1(1) = q_{1\nu}(1) = -t, \quad \nu = 1, 2, \\
q_1'(0) = -2 t_0, \quad q_1'(1) = -2 t_1, \quad q_1''(0) = -4t_0', \quad
q_1''(1) = -4t_1', \\
q_2(0) = t_0, \quad q_2(1) = t_1 + \tfrac{1}{2} t^2 - s, \quad
q_2'(0) = 2 t_0' - s_0, \quad
q_2'(1) = 2 t_1' - s_1 + 2 t_1 t, \\
q_{21}(0) = -t_0, \quad q_{21}(1) = \tfrac{1}{2} t^2- s - t_1, \quad
q_{22}(0) = -3t_0, \quad q_{22}(1) = -3t_1 + \tfrac{1}{2} t^2 - s, \\
q_3(0) = t_0' + s_0, \quad q_3(1) = t_1' + s_1 - t_1 t + u + ts - \tfrac{1}{6} t^3, \\
q_{31}(0) = 3 t_0', \quad q_{31}(1) = 3 t_1' + t_1 t + u + ts - \tfrac{1}{6} t^3, \\
q_{32}(0) = t_0' - s_0, \quad q_{32}(1) = t_1' - s_1 + 3 t_1 t + u + ts - \tfrac{1}{6} t^3.
\end{array} \right\}
\end{equation}

The eigenvalues $\la_{l,1}$ of the problem $\mathcal L_1$ \eqref{bc1} for equation \eqref{eqv4} coincide with the zeros of the characteristic determinant
$$
\Delta(\la) := \Delta_{1,1}(\la) = \begin{vmatrix}
\mathcal C_1(0,\la) & \mathcal C_2(0,\la) & \mathcal C_3(0,\la) & \mathcal C_4(0,\la) \\
\mathcal C_1''(1,\la) & \mathcal C_2''(1,\la) & \mathcal C_3''(1,\la) & \mathcal C_4''(1,\la) \\
\mathcal C_1'(1,\la) & \mathcal C_2'(1,\la) & \mathcal C_3'(1,\la) & \mathcal C_4'(1,\la) \\
\mathcal C_1(1,\la) & \mathcal C_2(1,\la) & \mathcal C_3(1,\la) & \mathcal C_4(1,\la)
\end{vmatrix}
$$

Therefore, the function $D(\rho)$ in \eqref{DelD} has the form
\begin{equation} \label{D41}
D(\rho) = \begin{vmatrix}
y_1(0,\rho) & y_2(0,\rho) & y_3(0,\rho) & y_4(0,\rho) \\
y_1''(1, \rho) & y_2''(1,\rho) & y_3''(1,\rho) & y_4''(1,\rho) \\
y_1'(1, \rho) & y_2'(1,\rho) & y_3'(1,\rho) & y_4'(1,\rho)  \\
y_1(1, \rho) & y_2(1,\rho) & y_3(1,\rho) & y_4(1,\rho)
\end{vmatrix}
\end{equation}

Substituting \eqref{asympty} for $n = 4$ together with \eqref{valq4} into \eqref{D41}, we obtain
\begin{align*}
& D(\rho) = \rho^2 \exp(\rho (\om_2 + \om_3 + \om_4)) d(\rho), \\
& d(\rho) = r_1(\rho) - r_2(\rho) \exp(\rho(\om_1 - \om_2)) + \eps(\rho), \\
& r_1(\rho) = 4i - \frac{4 i t}{\rho} + \frac{4 i\bigl(s + \tfrac{1}{2} t^2 + t_0 + t_1 \bigr)}{\rho^2} - \frac{4i (t_0' - t_1' + s_0 + s_1 + t(t_0 + t_1) + \tfrac{1}{6} t^3 + st - u)}{\rho^3},\\
& r_2(\rho) = -4 + \frac{4 i t}{\rho} + \frac{4 \bigl(s + \tfrac{1}{2} t^2 + t_0 + t_1 \bigr)}{\rho^2} - \frac{4i (t_0' - t_1' + s_0 + s_1 + t(t_0 + t_1) + \tfrac{1}{6} t^3 + st - u)}{\rho^3}.
\end{align*}

The zeros of $d(\rho)$ for large $|\rho|$ satisfy the relation \eqref{rho21}, which implies
\begin{multline*}
\rho_{l,1} = \frac{e^{-\frac{\pi i}{4}}}{\sqrt 2} \left( 2 \pi i l + \frac{\pi i}{2} + \frac{t(1-i)}{\rho_{l,1}} - \frac{2(s + t_0 + t_1)}{\rho_{l,1}^2} + \frac{(1 + i)(t_0' - t_1' + s_0 + s_1 - u)}{\rho_{l,1}^3} + \eps_l\right).
\end{multline*}
Hence
\begin{equation} \label{rho41}
\rho_{l,1} = e^{\frac{\pi i}{4}} \Biggl( \sqrt 2 \pi  l + \frac{\pi}{2\sqrt 2} - \frac{t}{\sqrt 2 \pi l + \frac{\pi}{2\sqrt 2}} + \frac{\sqrt 2(s + t_0 + t_1)}{\bigl( \sqrt 2 \pi l + \frac{\pi}{2 \sqrt 2}\bigr)^2} - \frac{t_0' - t_1' + s_0 + s_1 - u + t^2}{(\sqrt 2 \pi l)^3} + \eps_l\Biggr).
\end{equation}

Computing $\rho_{l,1}^4$ and taking \eqref{const4} into account, we arrive at the asymptotics \eqref{asymptla41} for $\la_{l,1}$.
 
In order to obtain the asymptotics for the eigenvalues $\mu_{l,1}$ of the boundary value problem $\mathcal M_1$ \eqref{eqv4}, \eqref{bcmu1}, we consider the functions
\begin{align}  \nonumber
\Delta^+(\la) := \Delta_{2,1}(\la) & = 
-\begin{vmatrix}
\mathcal C_1'(0,\la) & \mathcal C_2'(0,\la) & \mathcal C_3'(0,\la) & \mathcal C_4'(0,\la) \\
\mathcal C_1''(1,\la) & \mathcal C_2''(1,\la) & \mathcal C_3''(1,\la) & \mathcal C_4''(1,\la) \\
\mathcal C_1'(1,\la) & \mathcal C_2'(1,\la) & \mathcal C_3'(1,\la) & \mathcal C_4'(1,\la) \\
\mathcal C_1(1,\la) & \mathcal C_2(1,\la) & \mathcal C_3(1,\la) & \mathcal C_4(1,\la)
\end{vmatrix}, \\ \label{D41p}
D^+(\rho) & = \begin{vmatrix}
y_1'(0,\rho) & y_2'(0,\rho) & y_3'(0,\rho) & y_4'(0,\rho) \\
y_1''(1, \rho) & y_2''(1,\rho) & y_3''(1,\rho) & y_4''(1,\rho) \\
y_1'(1, \rho) & y_2'(1,\rho) & y_3'(1,\rho) & y_4'(1,\rho)  \\
y_1(1, \rho) & y_2(1,\rho) & y_3(1,\rho) & y_4(1,\rho)
\end{vmatrix}
\end{align}

Substituting \eqref{asympty} for $n = 4$ and \eqref{valq4} into \eqref{D41p}, we find
\begin{align*}
& D^+(\rho) = \rho^4 \exp(\rho(\om_2 + \om_3 + \om_4)) d^+(\rho), \\
& d^+(\rho) = r_1^+(\rho) - r_2^+(\rho) \exp(\rho(\om_1 - \om_2)) + \eps(\rho), \\
& r_1^+(\rho) = -4i + \frac{4it}{\rho} - \frac{4i\bigl( s + \frac{1}{2} t^2 + t_1 - t_0\bigr)}{\rho^2} + \frac{4i \bigl( 3 t_0' - t_1' + s_1 + t(t_1 - t_0) + \frac{1}{6} t^3 + st - u\bigr)}{\rho^3},\\
& r_2^+(\rho) = -4i - \frac{4t}{\rho} + \frac{4i \bigl(s + \frac{1}{2} t^2 + t_1 - t_0 \bigr)}{\rho^2} + \frac{4 \bigl(3 t_0' - t_1' + s_1 + t(t_1 - t_0) + \frac{1}{6} t^3 + st - u\bigr)}{\rho^3}.
\end{align*}

The zeros $\{ \varrho_{l,1}\}$ of $d^+(\rho)$ for large $l$ have the asymptotics
\eqref{rho21p}, which implies
\begin{align*}
\varrho_{l,1} & = \frac{e^{-\frac{\pi i}{4}}}{\sqrt 2} \left( 2 \pi i l + \frac{t(1-i)}{\varrho_{l,1}} - \frac{2(s - t_0 + t_1)}{\varrho_{l,1}^2} + \frac{(1 + i)(s_1 + 3 t_0' - t_1' - u)}{\varrho_{l,1}^3} + \eps_l\right) \\
& = 
e^{\frac{\pi i}{4}} \Biggl( \sqrt 2 \pi  l - \frac{t}{\sqrt 2 \pi l} + \frac{\sqrt 2(s - t_0 + t_1)}{(\sqrt 2 \pi l)^2} - \frac{s_1 + 3t_0' - t_1' - u + t^2}{(\sqrt 2 \pi l)^3} + \eps_l\Biggr).
\end{align*}
Computing $\mu_{l,1} = \varrho_{l,1}^4$ and taking \eqref{const4} into account, we arrive at \eqref{asymptmu41} for $\mu_{l,1}$.

Using the relations \eqref{calcbe1} and \eqref{smcalcbe}, which have the same form as in the case $n = 3$, we obtain
\begin{equation} \label{smbe1}
\be_{l,1} = -4 \la_{l,1} \left( 1 + \frac{i (t + 2t_0)}{\rho_{l,1}^2} + \frac{(1-i)(2s + s_0 + 2t_0 + 2t_1 - 2t_0')}{\rho_{l,1}^3} + \eps_l\right).
\end{equation}

Substituting \eqref{rho41} and \eqref{const4} into \eqref{smbe1}, we arrive at the asymptotics \eqref{asymptbe41} for $\be_{l,1}$. The formulas \eqref{asymptla41}, \eqref{asymptmu41}, and \eqref{asymptbe41} for $k = 3$ follow from the relations \eqref{relstar}: $\la_{l,3} = \la_{l,1}^{\star}$, $\mu_{l,3} = \mu_{l,1}^{\star}$, and $\be_{l,3} = \be_{l,1}^{\star}$, respectively. Indeed, for $n = 4$, equation \eqref{eqstar} has the form \eqref{eqv4} with the coefficients $\{ \tau_0, -\tau_1, \tau_2 \}$. Therefore, in formulas \eqref{asymptla41}, \eqref{asymptmu41}, and \eqref{asymptbe41}, only the sign of $\sigma$, $\sigma_0$, and $\sigma_1$ changes. For verification, the asymptotics for the case $k = 3$ have been explicitly derived in \cite{py}.

Similarly, we get the asymptotics of the eigenvalues $\la_{l,2}$ of the boundary value problem $\mathcal L_2$ \eqref{bc2}. They coincide with the zeros of the characteristic function
$$
\Delta(\la) := \Delta_{2,2}(\la) =
\begin{vmatrix}
\mathcal C_1(0,\la) & \mathcal C_2(0,\la) & \mathcal C_3(0,\la) & \mathcal C_4(0,\la) \\
\mathcal C_1'(0,\la) & \mathcal C_2'(0,\la) & \mathcal C_3'(0,\la) & \mathcal C_4'(0,\la) \\
\mathcal C_1'(1,\la) & \mathcal C_2'(1,\la) & \mathcal C_3'(1,\la) & \mathcal C_4'(1,\la) \\
\mathcal C_1(1,\la) & \mathcal C_2(1,\la) & \mathcal C_3(1,\la) & \mathcal C_4(1,\la)
\end{vmatrix}.
$$
Therefore, the function $D(\rho)$ in \eqref{DelD} has the form
\begin{equation} \label{D42}
D(\rho) =
\begin{vmatrix}
y_1(0,\rho) & y_2(0,\rho) & y_3(0,\rho) & y_4(0,\rho) \\
y_1'(0,\rho) & y_2'(0,\rho) & y_3'(0,\rho) & y_4'(0,\rho) \\
y_1'(1,\rho) & y_2'(1,\rho) & y_3'(1,\rho) & y_4'(1,\rho) \\
y_1(1,\rho) & y_2(1,\rho) & y_3(1,\rho) & y_4(1,\rho) 
\end{vmatrix}.
\end{equation}

Substituting \eqref{asympty} for $n = 4$ together with \eqref{valq4} into \eqref{D42}, we obtain
\begin{align*}
& D(\rho) = \rho^2 \exp(\rho(\om_3 + \om_4)) d(\rho), \\
& d(\rho) = r_1(\rho) - r_2(\rho) \exp(\rho(\om_2 - \om_3)) + \eps(\rho), \\
& r_1(\rho) = -2i - \frac{2t(1-i)}{\rho} + \frac{4\bigl(t_0 + t_1 + \frac{1}{2}t^2\bigr)}{\rho^2} - \frac{(1 + i) \bigl(6(t_1' - t_0') + 4t(t_0 + t_1) + \frac{2}{3} t^3 + 2u\bigr)}{\rho^3}\\
& r_2(\rho) = 2i - \frac{2t(1+i)}{\rho} + \frac{4\bigl(t_0 + t_1 + \frac{1}{2} t^2\bigr)}{\rho^2} - \frac{(1-i)\bigl(6(t_1' - t_0') + 4t(t_0 + t_1) + \frac{2}{3} t^3 + 2u\bigr)}{\rho^3}.
\end{align*}

The zeros of $d(\rho)$ for large $|\rho|$ satisfy the relation
$$
\rho = \frac{1}{\om_3 - \om_2} \ln \frac{r_2(\rho)}{r_1(\rho)} + \eps(\rho).
$$

Symbolic computations show that
$$
\rho_{l,2} = \pi l + \frac{\pi}{2} - \frac{t}{\rho_{l,2}} + \frac{2(t_1 + t_0)}{\rho_{l,2}^2} - \frac{3(t_1' - t_0') + u}{\rho_{l,2}^3} + \eps_l.
$$
Hence
\begin{equation} \label{rho42}
\rho_{l,2} = \pi l + \frac{\pi}{2} - \frac{t}{\pi l + \frac{\pi}{2}} + \frac{2(t_0 + t_1)}{\bigl(\pi l + \frac{\pi}{2} \bigr)^2} - \frac{3(t_1' - t_0') + u + t^2}{(\pi l)^3} + \eps_l.
\end{equation}

Computing $\rho_{l,2}^4$ and using \eqref{const4}, we arrive at the asymptotics \eqref{asymptla42}. 

For obtaining the asymptotics for the eigenvalues $\mu_{l,2}$, we have
\begin{equation} \label{D42p}
D^+(\rho) = 
\begin{vmatrix}
y_1(0,\rho) & y_2(0,\rho) & y_3(0,\rho) & y_4(0,\rho) \\
y_1''(0,\rho) & y_2''(0,\rho) & y_3''(0,\rho) & y_4''(0,\rho) \\
y_1'(1,\rho) & y_2'(1,\rho) & y_3'(1,\rho) & y_4'(1,\rho) \\
y_1(1,\rho) & y_2(1,\rho) & y_3(1,\rho) & y_4(1,\rho) 
\end{vmatrix}.
\end{equation}

Substituting \eqref{asympty} for $n = 4$ together with \eqref{valq4} into \eqref{D42p}, we obtain
\begin{align*}
& D^+(\rho) = \rho^3 \exp(\rho(\om_3 + \om_4)) d^+(\rho), \\
& d^+(\rho) = r_1^+(\rho) - r_2^+(\rho) \exp(\rho(\om_2 - \om_3)) + \eps(\rho), \\
& r_1^+(\rho) = 2 + 2i - \frac{4it}{\rho} - \frac{4(1-i)\bigl(t_1 + \frac{1}{2}t^2 \bigr)}{\rho^2} + \frac{4\bigl( 3 t_1' - t_0' + 2 t_1 t + \frac{1}{3} t^3 + u\bigr)}{\rho^3},\\
& r_2^+(\rho) = 2-2i + \frac{4it}{\rho} - \frac{4(1+i)\bigl( t_1 + \frac{1}{2} t^2 \bigr)}{\rho^2} + \frac{4\bigl(3 t_1' - t_0' + 2 t_1 t+ \frac{1}{3} t^3 + u \bigr)}{\rho^3}.
\end{align*}

The zeros of $d^+(\rho)$ for large $|\rho|$ satisfy the relation
$$
\rho = \frac{1}{\om_3 - \om_2} \ln \frac{r_2^+(\rho)}{r_1^+(\rho)} + \eps(\rho).
$$
Consequently, we obtain
\begin{align*}
\varrho_{l,2} & = \pi l + \frac{\pi}{4} - \frac{t}{\varrho_{l,2}} + \frac{2 t_1}{\varrho_{l,2}^2} - \frac{3 t_1' - t_0' + u}{\varrho_{l,2}^3} + \eps_l \\
& = \pi l + \frac{\pi}{4} - \frac{t}{\pi l + \frac{\pi}{4}} + \frac{2 t_1}{\bigl(\pi l + \frac{\pi}{4} \bigr)^2} - \frac{3 t_1' - t_0' + u + t^2}{(\pi l)^3} + \eps_l.
\end{align*}
Computing $\mu_{l,2} = \varrho_{l,2}^4$ and using \eqref{const4}, we arrive at \eqref{asymptmu42}.

It is known from the previous studies \cite{McL78, Pol23, Bond24-irkutsk} that the numbering in \eqref{asymptla41} and \eqref{asymptla42} starts from $l = 1$. The same conclusion for \eqref{asymptmu41} and \eqref{asymptmu42} follows from the explicit expressions of characteristic functions for $\tau = \{ 0, 0, 0\}$.

Proceed to estimating $\be_{l,2}$. Symbolic computations show that 
\begin{align*} 
\frac{d^+(\rho)}{\frac{d}{d\rho} d(\rho)} & = \frac{r_1^+(\rho) r_2(\rho) - r_2^+(\rho) r_1(\rho)}{\frac{d}{d\rho}r_1(\rho) r_2(\rho) - \frac{d}{d\rho} r_2(\rho) r_1(\rho) + (\om_3 - \om_2) r_1(\rho) r_2(\rho)} \\ \nonumber & = -\left( 1 + \frac{t + 2 t_0}{\rho^2} + \frac{4(t_0' - t_0 - t_1)}{\rho^3} + \eps(\rho)\right).
\end{align*}

Substituting \eqref{rho42} and using \eqref{calcbe1}, \eqref{const4}, we arrive at \eqref{asymptbe42}.

\section{Schur's test} \label{app:schur}

In this appendix, we present a discrete weighted version of Schur's test, which was originally proposed for integral operators in \cite{Schur11}. Applying Schur's test, we obtain an auxiliary estimate, which is used in the proof of \eqref{aD2} in Section~\ref{sec:sp}.

\begin{prop}[Schur's test] \label{prop:schur}
Suppose that $M_1, \, M_2 > 0$, $T_{k,n} \ge 0$, $u_n > 0$, and $v_n > 0$ for $k,n \ge 1$, and the following conditions are fulfilled:
\begin{enumerate}
\item $\sum\limits_{n = 1}^{\infty} T_{k,n} u_n \le M_1 v_k$, $k \ge 1$.
\item $\sum\limits_{k = 1}^{\infty} T_{k,n} v_k \le M_2 u_n$, $n \ge 1$.
\end{enumerate}
Then the operator $T$ given by $(T a)_k = \sum\limits_{n = 1}^{\infty} T_{k, n} a_n$ is bounded from $l_2$ to $l_2$ and
$\| T \|_{l_2 \to l_2} \le \sqrt{M_1 M_2}$.
\end{prop}

\begin{proof}
Let $\{ a_n \}_{n \ge 1} \in l_2$. Using the Cauchy-Bunyakovsky-Schwarz inequality, we obtain
$$
\left( \sum_{n = 1}^{\infty} T_{k,n} a_n \right)^2 = \left( \sum_{n = 1}^{\infty} \sqrt{T_{k,n} u_n} \sqrt{\frac{T_{k,n}}{u_n}} a_n \right)^2 \le \sum_{n = 1}^{\infty} T_{k,n} u_n \sum_{n = 1}^{\infty} \frac{T_{k,n}}{u_n} a_n^2.
$$
Applying the assumptions~1 and~2 and changing the summation order, we derive
$$
\sum_{k = 1}^{\infty} \left( \sum_{n = 1}^{\infty} T_{k,n} a_n \right)^2 \le \sum_{k = 1}^{\infty} M_1 v_k \sum_{n = 1}^{\infty} \frac{T_{k,n}}{u_n} a_n^2 = M_1 \sum_{n = 1}^{\infty} \frac{a_n^2}{u_n} \sum_{k = 1}^{\infty} T_{k, n} v_k \le M_1 M_2 \sum_{n = 1}^{\infty} a_n^2, 
$$
which proves the proposition.
\end{proof}

\begin{lem} \label{lem:bound}
The operator $T$ given by
$$
(T a)_k = \sum_{n = 1}^{\infty} \frac{a_n}{\ln(n + 1) \bigl( |n - k| + 1\bigr)}, \quad k \ge 1,
$$
is bounded from $l_2$ to $l_2$.
\end{lem}

\begin{proof}
Apply Schur's test with the weights $u_n = v_n = n^{-1/2}$. Divide the sum $\sum\limits_{n = 1}^{\infty} T_{k,n} u_n$ into three parts (with suitable rounding):
$$
\sum_{n = 1}^{\infty} = \sum_{n = 1}^{k/2} + \sum_{n = k/2}^{2k} + \sum_{n = 2k}^{\infty}
$$
and estimate them separately:
\begin{align*}
& \sum_{n = 1}^{k/2} \frac{1}{\ln(n + 1) \bigl( |n-k| + 1\bigr) \sqrt n} \le \frac{C}{k} \sum_{n = 1}^{k/2} \frac{1}{\sqrt{n}} \le \frac{C}{\sqrt k}, \\
& \sum_{n = k/2}^{2k} \frac{1}{\ln(n + 1) \bigl( |n-k| + 1\bigr) \sqrt n} \le \frac{C}{\ln(k+1) \sqrt k} \sum_{n = k/2}^{2k} \frac{1}{|n-k| + 1} \le \frac{C}{\sqrt k}, \\
& \sum_{n = 2k}^{\infty} \frac{1}{\ln(n + 1) \bigl( |n-k| + 1\bigr) \sqrt n} \le C \sum_{n = 2k}^{\infty} \frac{1}{n^{3/2}} \le \frac{C}{\sqrt k}.
\end{align*}
Analogously, we estimate the following sum:
\begin{align*}
& \sum_{k = 1}^{\infty} T_{k,n} v_k = \frac{1}{\ln(n+1)} \sum_{k = 1}^{\infty} \frac{1}{\bigl( |n-k| + 1\bigr)\sqrt k}, \\
& \sum_{k = 1}^{\infty} = \sum_{k = 1}^{n/2} + \sum_{k = n/2}^{2n} + \sum_{k = 2 n}^{\infty}, \\
& \sum_{k = 1}^{n/2} \frac{1}{\bigl( |n-k| + 1\bigr)\sqrt k} \le \frac{C}{n} \sum_{k = 1}^{n/2} \frac{1}{\sqrt k} \le \frac{C}{\sqrt n}, \\
& \sum_{k = n/2}^{2n} \frac{1}{\bigl( |n-k| + 1\bigr)\sqrt k} \le \frac{C}{\sqrt n} \sum_{k = n/2}^{2n} \frac{1}{|n-k| + 1} \le C \frac{\ln (n + 1)}{\sqrt n}, \\
& \sum_{k = 2n}^{\infty} \frac{1}{\bigl( |n-k| + 1\bigr)\sqrt k} \le C \sum_{k = 2n}^{\infty} \frac{1}{k^{3/2}} \le \frac{C}{\sqrt n}. 
\end{align*}
Thus, the conditions~1 and 2 of Proposition~\ref{prop:schur} are satisfied, so the operator $T$ is bounded.
\end{proof}

\begin{remark} \label{rem:ln}
The operator $T = [T_{k,n}]_{k,n \ge 1}$, $T_{k,n} = \dfrac{1}{|n-k|+1}$, is unbounded in $l_2$, which is shown by the following counterexample:
$$
a_n = \begin{cases}
            \sqrt n, & n \le N, \\
            0, & n > N.
      \end{cases}
$$
Indeed, one can easily check that $\dfrac{\| T a \|_{l_2}}{\| a \|_{l_2}} \to \infty$ as $N \to \infty$. 
Therefore, we have included an additional logarithm into \eqref{defTheta}.
Otherwise, the proof technique of Theorems~\ref{thm:scsp} and~\ref{thm:unisp} in Section~\ref{sec:sp} will not work.
\end{remark}

\medskip

{\bf Funding.} This work was supported by Grant 24-71-10003 of the Russian Science Foundation, https://rscf.ru/en/project/24-71-10003/.

\noindent Natalia Pavlovna Bondarenko \\

\noindent 1. Department of Mechanics and Mathematics, Saratov State University, 
Astrakhanskaya 83, Saratov 410012, Russia, \\

\noindent 2. Department of Applied Mathematics, Samara National Research University, 
Moskovskoye Shosse 34, Samara 443086, Russia, \\

\noindent 3. S.M. Nikolskii Mathematical Institute, RUDN University, 6 Miklukho-Maklaya St, Moscow, 117198, Russia, \\

\noindent e-mail: {\it bondarenkonp@sgu.ru}

\end{document}